\documentclass[final,3p,times]{elsarticle}
\usepackage{amssymb}
\usepackage{amsmath}
\usepackage{amsthm}

\allowdisplaybreaks

\usepackage{booktabs}
\usepackage{siunitx}

\usepackage{mathtools}

\numberwithin{equation}{section}

\usepackage{hyperref}
\hypersetup{
    colorlinks=true,
    linkcolor=blue,
    filecolor=magenta,      
    urlcolor=cyan,
    pdftitle={Overleaf Example},
    pdfpagemode=FullScreen,
    }

\newtheorem{thm}{Theorem}
\newtheorem{lem}[thm]{Lemma}
\newtheorem{prop}[thm]{Proposition}
\newtheorem{coro}{Corollary}
\newdefinition{rmk}{Remark}
\numberwithin{thm}{section}
\numberwithin{coro}{section}
\usepackage{etoolbox}
\makeatletter
\patchcmd{\ps@pprintTitle}% <cmd>
  {Preprint submitted}% <search>
  {To be submitted}% <replace>
  {}{}% <succes><failure>
\makeatother

\journal{}

\begin{document}

\DeclarePairedDelimiter{\abs}{\lvert}{\rvert} % | | absolute value
\DeclarePairedDelimiter{\norm}{\lVert}{\rVert} % || || norm
\DeclarePairedDelimiter{\rbra}{\lparen}{\rparen} % () round brackets
\DeclarePairedDelimiter{\cbra}{\lbrace}{\rbrace} % {} curly brackets
\DeclarePairedDelimiter{\sbra}{\lbrack}{\rbrack} % [] square brackets
\DeclarePairedDelimiter{\abra}{\langle}{\rangle} % < > angle brackets
\DeclarePairedDelimiter{\floor}{\lfloor}{\rfloor} % floor function
\DeclarePairedDelimiter{\ceil}{\lceil}{\rceil} % ceil function

\begin{frontmatter}


% use the tnoteref command within \title for footnotes;
% use the tnotetext command for theassociated footnote;
% use the fnref command within \author or \affiliation for footnotes;
% use the fntext command for theassociated footnote;
% use the corref command within \author for corresponding author footnotes;
% use the cortext command for theassociated footnote;
% \author{Name\corref{cor1}\fnref{label2}}
% \affiliation{organization={},

\title{Numerical Reconstruction of Coefficients in Elliptic Equations Using Continuous Data Assimilation}

\author[label1]{PEIRAN ZHANG}
\affiliation[label1]{organization={Graduate School of Informatics, Kyoto University},
            addressline={zhangpeiran@acs.i.kyoto-u.ac.jp}, 
            city={Yoshida-Honmachi, Sakyo-ku},
            postcode={606-8501}, 
            state={Kyoto},
            country={Japan}}

% \author[label2]{HIROSHI FUJIWARA}
% \affiliation[label2]{organization={Graduate School of Informatics, Kyoto University},
%             addressline={fujiwara@acs.i.kyoto-u.ac.jp },
%             city={Yoshida-Honmachi, Sakyo-ku},
%             postcode={606-8501},
%             state={Kyoto},
%             country={Japan}}

% \author{ZHANG Peiran} %% Author name

% \affiliation{organization={Kyoto University},%Department and Organization
%             addressline={zhang.peiran.76h@st.kyoto-u.ac.jp}, 
%             city={Kyoto},
%             postcode={606-8501}, 
%             state={Kyoto},
%             country={Japan}}

%% Abstract
\begin{abstract}
We consider the numerical reconstruction of the spatially dependent conductivity coefficient and the source term in elliptic partial differential equations in a two-dimensional convex polygonal domain, with the homogeneous Dirichlet boundary condition and given interior observations of the solution. Using data assimilation, we derive approximated gradients of the error functional to update the reconstructed coefficients. New $L^2$ error estimates are provided for the spatially discretized reconstructions. Numerical examples are given to illustrate the effectiveness of the method and demonstrate the error estimates. The numerical results also show that the reconstruction is very robust to the errors in specific inputted coefficients.
\end{abstract}


% \begin{highlights}
% \end{highlights}

%% Keywords
\begin{keyword}
coefficient reconstruction \sep continuous data assimilation \sep error estimate \sep elliptic partial differential equations

% MSC codes here, in the form: \MSC code \sep code
% or \MSC[2008] code \sep code (2000 is the default)
\end{keyword}

\end{frontmatter}



\section{Introduction}

We consider the problem of reconstructing the conductivity coefficient $q(x)$ and the source term $f(x)$ in elliptic partial differential equations with given spatially distributed solution data. To be specific, let $\Omega \subset \mathbb{R}^2$ be a convex polygonal domain, we consider recovering the coefficients $q(x)$ or $f(x)$ in
\begin{equation}
\begin{aligned}\label{eq:intro-1}
    - \nabla \cdot (q(x) \nabla u) + \nabla \cdot (b(x) u) + c(x) u &= f(x),  &&\text{in } \Omega \\
    u &= 0,  &&\text{on } \partial \Omega
\end{aligned}
\end{equation}
with an interpolant, $I_h u$, of $u$ given. The coefficients are assumed to be 
\begin{equation}\label{eq:para-assumption}
\begin{aligned}
    q &\in \mathcal{Q} := \{q \in H^1(\Omega): 0<q_0 \leq q(x) \leq q_1 \;a.e.\},\\
    b &= (b_1, b_2)^\intercal \text{ with } b_j \in H^1(\Omega), \, j=1,2,\\
    c &\in L^\infty (\Omega), \,\text{with}\, 0\leq  \,c(x) \leq c_1 \; a.e.,\\
    f &\in L^2(\Omega).
\end{aligned}
\end{equation}
Under these conditions (\ref{eq:intro-1}) admits a unique weak solution $u\in H^1_0(\Omega)$ (see, e.g. \cite{evans2010partial}).

Equation (\ref{eq:intro-1}) covers a wide range of practical applications. For example, when $b = c=0$ (\ref{eq:intro-1}) may correspond to the steady state of the model for the percolation of groundwater in a confined aquifer, with $q$ the hydraulic conductivity of the aquifer, $u$ the piezometric head and $f$ the recharge (\cite{knowles1996variational}, \cite{anderson2015applied}). When $b = 0$, (\ref{eq:intro-1}) is the equation of the same time-dependent model with the Laplace transform applied to $u$ \cite{knowles2001parameter}. When $c = 0$, (\ref{eq:intro-1}) is the steady state of the transport equation that models, for example, the distribution of biological species with $q$ the diffusion coefficient\cite{banks2012estimation}. For many other applications, we refer to \cite{banks2012estimation}.

The method proposed in this work regards (\ref{eq:intro-1}) as the steady state of the corresponding time-dependent parabolic equation, allowing the techniques to be applied naturally to time-dependent problems. To be specific, we treat (\ref{eq:intro-1}) as the steady state of the parabolic equation
\begin{equation}\label{eq:intro-2}
\begin{aligned}
    u_t - \nabla \cdot (q \nabla u) + \nabla \cdot (b u) + c u &= f,  &&\text{in } (0,T] \times \Omega \\
    u(0) &= u_0,  &&\text{in } \Omega\\
    u &= 0,  &&\text{on }  [0,T] \times \partial \Omega
\end{aligned}
\end{equation}
and we shall utilize the continuous data assimilation (CDA) equation \cite{azouani2014continuous}, \cite{carlson2020parameter}
\begin{equation}\label{eq:intro-3}
\begin{aligned}
    v_t - \nabla \cdot (\tilde{q} \nabla v) + \nabla \cdot (\tilde{b} v) + \tilde{c} v &= \tilde{f} + \mu I_h (u-v),  &&\text{in } (0,T] \times \Omega\\
    v(0) &= v_0,  &&\text{in } \Omega\\
    v &= 0,  &&\text{on } [0,T] \times \partial \Omega
\end{aligned}
\end{equation}
where $I_h u$ is given with $I_h$ an $L^2$ projection on a discrete mesh of mesh size $h$, $\mu$ is a positive relaxation parameter, and $\mu I_h (u-v)$ is called a feedback term. 

CDA has been used in the constant coefficient recovery by \cite{carlson2020parameter} and \cite{newey2025model}. In this work, we extend the application of CDA in coefficient recovery from constant coefficients to spatially dependent ones. CDA essentially leverages the observation $I_h u$ and can be classified as a nudging method in data assimilation. Conventionally, nudging methods are designed for finite-dimensional systems \cite{asch2016data}. In \cite{azouani2014continuous}, the technique was introduced to dynamical systems in infinite-dimensional spaces and is known as the AOT (Azouani-Olson-Titi) algorithm in the literature. In \cite{carlson2020parameter}, this method is further used to identify a constant viscosity parameter in 2 dimensional Navier-Stokes equations, and was recently in \cite{newey2025model} extended to the recovery of constant parameters in other types of equations or systems (e.g. the Lorenz '63 model and Kuramoto-Sivashinsky equation), and the authors provided with us a framework to obtain the gradients of an input-output error functional conveniently thanks to the feedback term $\mu I_h (u-v)$.  For works on continuous data assimilation (the AOT algorithm) and coefficient recovery, we also refer to \cite{carlson2021dynamically}, \cite{martinez2022convergence} for convergence analysis for the recovery of constant parameters, \cite{ccibik2025data} for assimilation with errors in model taken into account, and \cite{yushutin2025continuous} for CDA using mean value and boundary projection as the feedback term.

% Finding $u$ in (\ref{eq:intro-1}) with all other coefficients and conditions specified is called the direct or forward problem, while our interested problem of determining a coefficient in (\ref{eq:intro-1}) given an approximation $I_h u$ of the solution $u$ is called the inverse problem. Typical difficulties in such a problem include that the given information of $u$ is made by experimental measurement in practice and is thus noisy, so there is usually no solution to (\ref{eq:intro-1}) can match it exactly, and that there are more than one choice of coefficients that produce almost the same solutions which are very close to $I_h u$, and the coefficient-to-solution map from $q$ or $f$ to $u$ can be very sensitive so that in computations small error in the given $I_h u$ may yield a completely different identification result. In addition, the problem of identifying $q$ is a nonlinear one since $q \mapsto u$ is, even the underlying equation (\ref{eq:intro-1}) is a linear PDE. For general references to inverse problems, see, for example \cite{hasanouglu2021introduction}, \cite{kirsch2011introduction}. In our work, the identification for $f$ is mainly used to demonstrate the extensibility of the method to other coefficients or equations.

Another novelty of this article is the new $L^2$ error estimates for the proposed reconstruction method using CDA. With the inspiration from \cite{bonito2017diffusion} and \cite{hoffmann1985identification}, we derive a sub-linear $L^2$ error estimate for the discretized reconstruction of $q\in H^m(\Omega)$ ($m\geq 1$), and a linear $L^2$ error estimate for the discretized reconstruction of $f \in L^2(\Omega)$. In particular, the parameter $\mu$ influences the reconstruction error in a manner similar to that of the regularization parameter in the classical Tikhonov regularization (see, e.g. \cite{kirsch2011introduction}), and a wide range of values $\mu$ produces satisfactory reconstructions, so there is no difficulty choosing a suitable $\mu$. In addition, the proposed error estimates account for errors in both solution data and coefficients of the PDE model, which is rarely seen in the present works on this topic. Until now, there exists a vast amount of literature investigating the problem of coefficient reconstruction in elliptic equations with Dirichlet or Neumann boundary conditions and interior measurement of $u$, for example \cite{cen2024numerical}, \cite{jin2021error}, \cite{acar1993identification}, \cite{hanke1997regularizing}, \cite{knowles2001parameter}, \cite{chan2003identification}, \cite{al2012stability}, \cite{kohn1988variational}, and in \cite{bonito2017diffusion} the authors present a Hölder stability of the reconstruction problem for $q$ when $b = c = 0$ under suitable regularity of the domain $\Omega$. The conventional reconstruction methods include (i) minimizing an output least squares error, that is, minimizing $\|v - K u\|$ for some observation $K u$ of $u$ in a norm $\|\cdot\|$, (ii) minimizing an equation error, that is, minimizing $\Vert\nabla \cdot (\tilde{q} \nabla v) + f\Vert$ in the case $b=c=0$ for example, (iii) introduce a special error functional with certain desirable properties and minimize it. 
% (\cite{knowles2001parameter}, \cite{kohn1988variational}). 
There are also authors who treat the origin problem as a first-order hyperbolic equation for $q$, instead of an equation for $u$ \cite{richter1981numerical}. In practical computations, regularization terms should usually be added to the objective function by an output error or an equation error to avoid instability, and discretizations should be made to obtain an approximation $\tilde{q}_{h_q}$ or $\tilde{f}_{h_f}$ ($h_q$ or $h_f$ is the mesh size for the discrete reconstruction). However, in the aspect of error estimates, the regularization terms and the discretization of coefficients are not always considered, and explicit estimates for the relationship between the $L^2$ norm of the error $\tilde{q}_{h_q} - q$ and the noise in the input are difficult. For example, in \cite{falk1983error}, an error $L^2$ estimate for the discretized reconstruction of the coefficient is derived for an equation error method without regularization; in \cite{kohn1988variational}, assuming the $q \in H^2 (\Omega)$ and $u \in H^3(\Omega)$ for an special error functional and the resulting discretized reconstruction, an $L^2$ estimate is derived in the case of no regularization and an weighted estimate of $\int_\Omega |q_{h_q}-q| |\nabla u|^2$ is derived in the case of an $H^1$ regularization term added; in \cite{jin2021error} and \cite{cen2024numerical}, $L^2$ estimates for the discretized reconstruction with $H^1$ regularization terms are derived, but in a form of $\|\tilde{q}_{h_q} - q\|\leq c\delta^\frac{1}{4(1+\beta)}$ where $\delta$ is the noise level and $\beta \geq 0$ is the constant in a \textit{positivity condition} (see also \cite{bonito2017diffusion}) that depends on the regularity of the underlying domain, and this convergence rate is reported to be suboptimal in \cite{jin2021error}. For the positivity condition, let us mention that we shall also use a positivity condition in this work to prove the error estimates, but it can be verified directly using $u$ without the knowledge of the coefficients, unlike that in \cite{bonito2017diffusion} or \cite{cen2024numerical}, which involves $q$ and $f$. 

We should also emphasize that the proposed approach avoids numerically solving an adjoint problem to obtain the gradient of $J$ as in many conventional methods (see \cite{hasanouglu2021introduction} for examples). In this work, the solution $v$ of (\ref{eq:intro-3}) will be shown to be an approximation to the solution $u$ of (\ref{eq:intro-2}) for relatively large $\mu$, and in particular this approximation property is also true for the steady state, i.e., the elliptic equation (\ref{eq:intro-1}) that we want to investigate. To reconstruct the coefficients $q$ and $f$, an easily computed approximated gradient formula for the error functional $J = \frac{1}{2}\| I_h u - I_h v_{\tilde{q} , \tilde{f}} \|_{L^2}^2$ is derived, where $v_{\tilde{q}, \tilde{f}}$ solves
\begin{equation}\label{eq:intro-4}
\begin{aligned}
    - \nabla \cdot (\tilde{q} \nabla v) + \nabla \cdot (\tilde{b} v) + \tilde{c} v &= \tilde{f} + \mu I_h (u-v),  &&\text{in } \Omega \\
    v &= 0,  &&\text{on } \partial \Omega
\end{aligned}
\end{equation}
and the gradient formula is used to update the coefficient $\tilde{q}$ or $\tilde{f}$ in the finite element spaces, and we do not need to compute an adjoint problem to obtain the gradient. It is also interesting that the approximated gradient formula for the reconstruction of $q$ turns out to be the same up to a constant as that in \cite{hoffmann1985identification}, although the approach in \cite{hoffmann1985identification} was not identified as a minimization for some error functional. In \cite{hoffmann1985identification}, the author embeds the elliptic equation into a parabolic equation by adding an $\varepsilon u_t$ to (\ref{eq:intro-1}), while in our case (\ref{eq:intro-1}) remains time-independent in numerical computations, but a feedback term is added to obtain an additional assimilation equation. 

Numerical examples are provided as well, and they on the one hand verify our theoretical error estimates, and on the other hand exhibit a notable feature of this method that a most rough estimation in $b$ and $c$ (suitable constants, for example) used as the inputted $\tilde{b}$ and $\tilde{c}$ still yields decent reconstruction accuracy. The latter is helpful for real-world coefficient reconstructions, since in many cases we have little knowledge about $b$ and $c$. However, the numerical experiments also reveal that the reconstructions using fine meshes tend to be unstable. 

We also emphasize that, although the target of our current work is coefficient reconstruction in elliptic equations, the essentially same approach can be naturally applied to the corresponding parabolic equations, as it has already been seen that our analysis departs from them.

The remainder of the paper is organized as follows. 
% In Section \ref{sect2}, we introduce some notation and recall some facts on function spaces. 
In Section \ref{sect3}, the $L^2$ norm and $H^1$ semi-norm error estimates for the assimilation are proved for both elliptic and parabolic equations, and some useful properties, mainly the Lipschitz continuity, of the coefficient-to-solution maps are discussed. In Section \ref{sect4}, approximated gradients of our error functional are derived to update the reconstructed coefficients, and error estimates for such minimization formulations, which do not use exact gradients in updates, are given. In Section \ref{sect5}, numerical results are shown for the reconstruction of $q$ and $f$.

% \section{Preliminaries}\label{sect2}

We end the introduction by briefly stating some notation and recalling some results for the polynomial approximation properties. Recall that $\Omega$ is set to be a convex polygon. We denote by $\| \cdot \|_m$ and $|\cdot|_m$ the norms and the semi-norms of $H^m(\Omega)$ where $m\geq 0$, and by $\|\cdot\| = \|\cdot\|_0$ the $L^2(\Omega)$ norm. The dual space of $H^1_0(\Omega)$ is denoted $H^{-1}(\Omega)$, with the norm $\|u\|_{-1} = \sup_{\varphi\in H_0^1(\Omega)} \int_\Omega u \varphi \mathrm{d} x/\Vert\varphi\Vert$ for $u \in H^{-1}(\Omega)$. All these spaces are considered over $\mathbb{R}$. The inner product in $L^2(\Omega)$ will then be written as $(f, g) = \int_\Omega f g \mathrm{d} x$ for $f, g \in L^2(\Omega)$. For $u \in H_0^1 (\Omega)$, we have the Poincar\'e's inequality $\|u\| \leq C_P |u|_1$ where we denote the constant by $C_P$. We also assume the elliptic regularity for the solution $u$ of (\ref{eq:intro-1}) with the domain $\Omega$ and the coefficients that we consider, that is, we assume $\|u\|_2 \leq C \|f\|$ for the solution $u$ of (\ref{eq:intro-1}) and some $C>0$. In the following parts, we shall drop $\mathrm{d}x$ in the integrals if no confusion may occur.

Let $\mathcal{T}_{h}$  ($h<1$ the mesh size) be a quasi-uniform (\cite{brenner2008mathematical}) subdivision of $\Omega$, and for $\mathcal{T}_h$ let $\mathcal{P}_k (\mathcal{T}_{h})$ be the space of functions that is a Lagrange polynomial of order $k$ in each $T \in \mathcal{T}_h$. Denote by $I_{h}: L^2 (\Omega) \to V_h = \mathcal{P}_k (\mathcal{T}_{h})$ the $L^2$ projection satisfying
\begin{equation*}
    (I_h w - w, v_h) = 0, \quad \forall w\in L^2(\Omega),\, v_h \in V_h.
\end{equation*}
 We have the approximation property \cite{brenner2008mathematical},
\begin{equation*}
    \|I_{h} u - u \|_s \leq C_{A} h^{m-s} | u |_m, \quad u \in H^m(\Omega) , \quad 0\leq s\leq m \leq k.
\end{equation*}
% Let $I_h: L^2(\Omega)\to V_h$ to be a quasi-interpolation operator that is an $L^2$ projection:
% \begin{equation*}
%     (I_h w - w, v_h) = 0, \quad \forall w\in L^2(\Omega),\, v_h \in X_h,
% \end{equation*}
% we have the interpolation error estimates, 
% \begin{equation*}
%     \|I_{h} u - u \|_s \leq C_{A} h^{m-s} \| u \|_m, \quad u \in H^m(\Omega), \quad 0\leq s\leq m.
% \end{equation*}
And we have the inverse estimate in the finite-dimensional space $V_h$ \cite{brenner2008mathematical}
\begin{equation}\label{eq:inverse-estimate}
    \|v_h \|_1 \leq C_{\text{inv}} h^{-1} \| v_h \|, \quad v_h \in V_h,
\end{equation}
where $C_{\text{inv}}$ depends on the regularity of the subdivisions.

\section{Estimates for assimilation error and the some properties of the coefficient-to-solution maps}\label{sect3}

\subsection{Estimates for assimilation error}

Let $u$, $v$ be the solutions of
\begin{equation}
\begin{aligned}\label{eq:plm-1}
    u_t - \ \nabla \cdot \rbra[\big]{q(x) \nabla u } + \nabla \cdot \bigl( b(x) u \bigr) + c(x) u &= f(t, x), &&\text{in } (0, T] \times \Omega \\
    u &= 0,  &&\text{on } [0, T] \times \partial \Omega  \\
    u(0) &= u_0, &&
\end{aligned}    
\end{equation}
and
\begin{equation}
\begin{aligned}\label{eq:plm-2}
    v_t - \nabla \cdot \Bigl(\bigl(q(x) + \delta q(x)\bigr) \nabla v\Bigr) + \nabla \cdot \Bigl( \bigl(b(x) + & \delta b(x)\bigr) v \Bigr) + \bigl(c(x) + \delta c(x)\bigr) v &&\\
    &= f(t, x) + \delta f(t, x) + \mu I_h (u - v) \quad &&\text{in } (0, T] \times \Omega  \\
    v &= 0 &&\text{on } [0, T] \times \partial \Omega \\
    v(0) &= v_0, &&   
\end{aligned}  
\end{equation}
where $q$, $q + \delta q \in \mathcal{Q}$, $b_j, \delta b_j\in H^1(\Omega)$, $c, c + \delta c$ are non-negative a.e. with an essential supremum $c_1$, and $f, \delta f \in L^2 (0, T; L^2 (\Omega))$, and $I_h: L^2(\Omega) \to \mathcal{P}_k(\mathcal{T}_h)$ with $k \geq 1$ is an $L^2$ projection. Note that a more general case that $f$ can depend on $t$ is treated here, while in (\ref{eq:intro-2}) and (\ref{eq:intro-3}) the $f$ is independent of $t$. The unique solvability of (\ref{eq:plm-1}) can be found in, e.g. \cite{evans2010partial}. The existence of a unique solution to (\ref{eq:plm-2}) should be able to be verified using a Galerkin approximation technique (\cite[Chapter 7]{evans2010partial}). 

It is naturally expected that $v(t)$ converges to $u(t)$ in $L^2$ for large $t$ and $\mu$ if all coefficients and the source term in (\ref{eq:plm-2}) are taken to be exact, but inspired by \cite{ccibik2025data}, we see that the approximation is valid even if errors are presented in the coefficients and the source term in (\ref{eq:plm-2}).

\begin{lem}\label{lemma1}
    Denote the error by $w = u - v$. For $t>0$, if $\mu$ is taken large enough and $h$ is taken small enough such that $(\frac{3}{4}-2C_P^{-2}C_A^{2}h^2)\mu \geq 2C_P^{-2}(b_1 C_P -q_0)-2c_0$, we have
    \begin{align}\label{eq:assimilation-err-1}
    \lVert w(t) \rVert^2 \leq & \lVert w(0) \rVert^2 e^{-\mu t/4} \nonumber\\
    & \quad + \frac{4}{\mu} \int_0^t \Bigl(\lVert\delta f(s) \rVert^2 + C_P^{-2} \lVert \delta b u(s) \rVert^2  + \lVert \delta c u(s) \rVert^2 + \norm{ \nabla \cdot(\delta q \nabla u(s)) }_{-1}\Bigr) e^{\mu (s - t)/4} \mathrm{d} s.
    \end{align}
\end{lem}

\begin{proof}
    We denote $\tilde{q} = q+\delta q$, $\tilde{b} = b+\delta b$ and $\tilde{c} = c+\delta c$. Since $w = u - v$, it satisfies 
    \begin{equation}\label{eq:assimilation-err-1}
    w_t - \nabla \cdot (\tilde{q} \nabla w) + \nabla \cdot (\tilde{b} w) + \tilde{c} w + \mu I_h w = - \delta f - \nabla \cdot (\delta q \nabla u) +  \nabla \cdot (\delta b u) + \delta c u,     
    \end{equation}
    with $w=0$ on $\partial \Omega$. Taking $L^2$ inner product using (\ref{eq:assimilation-err-1}) with $w$, we have
    \begin{align*}
    & \frac{1}{2} \frac{\mathrm{d}}{\mathrm{d} t} \norm{w}^2 + (\tilde{q} \nabla w, \nabla w) + (\nabla \cdot (\tilde{b} w), w) + (\tilde{c} w, w) + \mu (I_h w, w) \\
    =\,& (- \delta f, w) - (\nabla \cdot (\delta q \nabla u), w) + (\delta b u, \nabla \cdot w) + (\delta c u, w).
    \end{align*}
    Since $I_h$ is assumed to be a $L^2$ projection, we have
    \begin{align*}
    (I_h w, w) = (I_h w, w - I_h w + I_h w) = (I_h w, I_h w) = \norm{I_h w}^2,
    \end{align*}
    and also, using $\norm{w - I_h w} \leq C_A h \abs{w}_1$,
    \begin{equation*}
    \norm{w}^2 = \norm{w - I_h w}^2 + \norm{I_h w}^2 \leq C_A^2 h^2 \abs{w}^2_1 + \norm{I_h w}^2.
    \end{equation*}
    Then $q \geq q_0$ and $c\geq c_0$ yield
    \begin{align*}
    & \frac{\mathrm{d}}{\mathrm{d} t} \norm{w}^2 + 2 q_0 \abs{w}_1^2 + 2 c_0 \norm{w}^2 +  2 \mu\rbra[\Big]{-C_A^2 h^2 \abs{w}_1^2 + \norm{w}^2} \\
    \leq \,& 2 \rbra[\Big]{ \norm{\delta f} + \norm{\delta c u} } \norm{w} + 2 \norm{\nabla \cdot(\delta q \nabla u)}_{-1} \|w\|_1 + 2\norm{\delta b u} \abs{w}_1  +2 b_1 \norm{w} \abs{w}_1,
    \end{align*}
    which gives
    \begin{align*}
    & \frac{\mathrm{d}}{\mathrm{d} t} \norm{w}^2 + \rbra[\big]{2 q_0 - 2 \mu h^2 C_A^2} \abs{w}_1^2 + (2 c_0 + 2 \mu) \norm{w}^2 \\
    \le & \:\frac{4}{\mu} \Bigl(\|\delta f\|^2 + \|\delta c u\|^2 + \|\nabla \cdot(\delta q \nabla u)\|_{-1}^2\Bigr) + \frac{4}{\mu C_P^2} \|\delta b u\|^2 + \frac{\mu}{2} \|w\|^2  + \frac{1}{4}\mu \|w\|_1^2 + \frac{C_P^2}{4}\mu |w|_1^2 + 2b_1 C_P |w|_1^2\\
    \le & \:\frac{4}{\mu} \Bigl(\|\delta f\|^2 + \|\delta c u\|^2 + \|\nabla \cdot(\delta q \nabla u)\|_{-1}^2\Bigr) + \frac{4}{\mu C_P^2} \|\delta b u\|^2  + \frac{\mu}{2} \|w\|^2 + \frac{3C_P^2}{4}\mu |w|_1^2 +2 b_1 C_P |w|_1^2
    \end{align*}
    Using Poincar\'e's inequality with $C_P$ the constant in the inequality, we then have
    % \begin{align*}
    % & \frac{\mathrm{d}}{\mathrm{d} t} \|w\|^2 + C_P^{-2} \left(2 q_0 - 2 \mu h^2 C_A^2  - 2b_1 C_P - \frac{3C_P^2 \mu}{4} \right) \|w\|^2 + \left(2 c_0 + 2 \mu - \frac{3 \mu}{4} \right) \norm{w}^2 \\
    % \le &\: \frac{4}{\mu} \rbra[\Big]{\norm{\delta f}^2 + C_P^{-2} \norm{\delta b u}^2  + \norm{\delta c u}^2 + \norm{\nabla \cdot(\delta q \nabla u)}_{-1}^2}
    % \end{align*}
    % i.e.,
    \begin{align*}
    & \frac{\mathrm{d}}{\mathrm{d} t} \norm{w}^2 + \left[2 C_P^{-2}(q_0 - \mu C_A^2 h^2 - b_1 C_P) + 2 c_0 + \frac{1}{2}\mu\right] \norm{w}^2 \\
    \le & \: \frac{4}{\mu} \rbra[\Big]{\norm{\delta f}^2 + C_P^{-2} \norm{\delta b u}^2  + \norm{\delta c u}^2 + \norm{\nabla \cdot(\delta q \nabla u)}_{-1}}
    \end{align*}
    Thus if $\mu$ and $h$ fulfill $2C_P^{-2} \rbra[\big]{q_0 - \mu C_A^2 h^2 - b_1 C_P} + 2c_0 + \mu\geq \frac{1}{4}\mu$, we will have
    \begin{align*}
    \frac{\mathrm{d}}{\mathrm{d} t} \norm{w}^2 + \frac{1}{4}\mu\norm{w}^2 & \le \frac{4}{\mu} (\norm{\delta f}^2 + C_P^{-2} \norm{\delta b u}^2  + \norm{\delta c u}^2 + \norm{\nabla \cdot(\delta q \nabla u)}_{-1}).
    \end{align*}
    Now integrating both sides leads to the conclusion.  

\end{proof}

\begin{rmk}
    As can be seen from the proof, lower regularities for the coefficients with $q\in L^\infty(\Omega)$ ($q_0 \leq q \leq q_1$ a.e.) and $b_j \in L^\infty (\Omega)$ is enough to prove Lemma \ref{lemma1}. 
\end{rmk}

Assuming further regularity of the solutions, we have the next estimate for the assimilation error in the $H^1$ semi-norm. 

\begin{lem}\label{lemma2}
    Assume that $u, v \in L^2(0, T; H^2 (\Omega) \cap H^1_0(\Omega))$ and $\frac{\mathrm{d}}{\mathrm{d}t} u, \frac{\mathrm{d}}{\mathrm{d}t} v \in L^2(0, T; H^1_0(\Omega))$. For $t>0$, if $\mu \geq 2 + 4\rbra[\Big]{(1+C_P)^2 \|b+\delta b\|_1^2 + 2c_1 C_P + |q+\delta q|_1^2}/q_0$, and $h^2 \leq C_A^2q_0/2\mu^2$, then we have for $w = u-v$ the error between solutions of (\ref{eq:plm-1}) and (\ref{eq:plm-2}) that
    \begin{align}\label{eq:lemma2-est}
        |w(t)|_1^2 \leq |w(0)|_1^2 e^{-\mu t/2} &+ \frac{1}{\mu} \int_0^t  |\delta c u(s)|_1^2 e^{\mu(s-t)/2} \mathrm{d}s \nonumber\\
        &+ \frac{3}{q_0} \int_0^t  \bigl(|\delta b u(s)|_1^2 + \|\nabla \cdot (\delta q \nabla u(s))\|^2 + \|\delta f(s)\|^2\bigr) e^{\mu(s-t)/2} \mathrm{d}s.
    \end{align}
\end{lem}

\begin{proof}
    Denote $\tilde{q} = q+\delta q$, $\tilde{b} = b+\delta b$ and $\tilde{c} = c+\delta c$. Then $w$ satisfies
    \[
    w_t - \nabla \cdot (\tilde{q} \nabla w) + \nabla \cdot (\tilde{b} w) + \tilde{c} w + \mu I_h w = - \delta f - \nabla \cdot (\delta q \nabla u) + \nabla \cdot (\delta b u) + \delta c u.
    \]
    By integrating both sides with $\Delta w$ and using Cauchy-Schwartz's inequality, we obtain
    \begin{align*}
        &\frac{1}{2} \frac{\mathrm{d}}{\mathrm{d} t} \abs{w}_1^2 + (\nabla \cdot (\tilde{q} \nabla w), \Delta w) - (\nabla \cdot (\tilde{b} w), \Delta w) - (\tilde{c} w, \Delta w) \nonumber\\
        \leq &\:\norm{\delta f} \norm{\Delta w} + \norm{\nabla \cdot(\delta q \nabla u)} \norm{\Delta w} + |\delta b u|_1 \|\Delta w\| + \abs{\delta c u}_1 \abs{w}_1 + \mu (I_h w, \Delta w), \nonumber
    \end{align*}
    thus
    \begin{equation}\label{eq:lemma2-1}
    \begin{aligned}
        \frac{1}{2} \frac{\mathrm{d}}{\mathrm{d} t} \abs{w}_1^2 + (\tilde{q} \Delta w, \Delta w) &\leq \Bigl( \norm{\delta f} +|\delta b u|_1 + \norm{\nabla \cdot(\delta q \nabla u)} \Bigr)  \norm{\Delta w} + \abs{\delta c u}_1 \abs{w}_1 \\
        &\qquad + |(\nabla \cdot (\tilde{b}w), \Delta w)| + \abs{(\tilde{c} w, \Delta w)} + \abs{(\nabla \tilde{q} \cdot \nabla w, \Delta w)} + \mu (I_h w, \Delta w).
    \end{aligned}
    \end{equation}
    Then for $(I_h w, \Delta w)$ we have
    \begin{align*}
        (I_h w, \Delta w) = (I_h w - w, \Delta w) + (w, \Delta w) &\leq C_A h |w|_1 \| \Delta w \| - |w|_1^2 \leq \rbra[\Big]{\frac{1}{2\mu}-1} |w|_1^2 +\frac{C_A^2 h^2 \mu}{2} \| \Delta w \|^2,
    \end{align*}
    and we have respectively for $|\delta c u|_1 |w|_1$, $|(\nabla \tilde{q} \cdot \nabla w, \Delta w)|$, $|(\tilde{c}w, \Delta w)|$ and $|(\nabla \cdot (\tilde{b}w), \Delta w)|$ that,
    \begin{align*}
        |\delta c u|_1 |w|_1 &\leq \frac{1}{2\mu} |\delta c u|_1^2 + \frac{\mu}{2} |w|_1^2,\\
        |(\nabla q \cdot \nabla w, \Delta w)| & \leq |q|_1 |w|_1 \|\Delta w\| \leq \frac{|\tilde{q}|_1^2}{q_0}|w|_1^2 + \frac{q_0}{4}\|\Delta w\|^2,\\
        |(\tilde{c}w, \Delta w)| &\leq c_1 \|w\|\|\Delta w\| \leq C_P^2 \frac{c_1}{q_0} + \frac{q_0}{4} \|\Delta w\|^2,\\
        |(\nabla \cdot (\tilde{b} w), \Delta w)| & \leq \|\tilde{b}\|_1 \|w\|_1 \|\Delta w\| \leq \frac{(1+C_P)\|\tilde{b}\|_1^2}{q_0} |w|_1^2 + \frac{q_0}{8}\|\Delta w\|^2.
    \end{align*}
    Hence from (\ref{eq:lemma2-1}) we have, with some rearrangements,
    \begin{align*}
    &\frac{1}{2} \frac{\mathrm{d}}{\mathrm{d} t} |w|_1^2 + q_0 \|\Delta w\|^2\\
        \leq &\frac{1}{2\mu} |\delta c u|_1^2 + \Bigl(|\delta b u|_1 + \|\nabla \cdot (\delta q \nabla u)\| + \|\delta f\|\Bigr) \|\Delta w\| \\
        &\qquad+\Bigl(\frac{\mu}{2} + \frac{(1+C_P)^2 \|\tilde{b}\|_1^2 + 2c_1 C_P + |\tilde{q}|_1^2}{q_0} + \frac{1}{2} - \mu\Bigr)|w|_1^2 + \Bigl(\frac{1}{2}q_0 + \frac{C_A^2 h^2 \mu^2}{2}\Bigr)\|\Delta w\|^2\\
        \leq& \frac{1}{2\mu} |\delta c u|_1^2 + \frac{3}{q_0} \bigl(|\delta b u|_1^2 + \|\nabla \cdot (\delta q \nabla u)\|^2 + \|\delta f\|^2\bigr) \\
        &\qquad+\Bigl(-\frac{\mu}{2}  + \frac{1}{2} + \frac{(1+C_P)^2 \|\tilde{b}\|_1^2 + 2c_1 C_P + |\tilde{q}|_1^2}{q_0})|w|_1^2 + \Bigl(\frac{3}{4}q_0 + \frac{C_A^2 h^2 \mu^2}{2}\Bigr)\|\Delta w\|^2.
    \end{align*}
    Thus if we pose the conditions $C_A^2 h^2 \mu^2/2 \leq q_0/4$ and $-\mu/2 + \rbra[\Big]{(1+C_P)^2 \|\tilde{b}\|_1^2 + 2c_1 C_P + |\tilde{q}|_1^2}/q_0 + 1/2 \geq -\mu/4$ (namely, $\mu \geq 4\rbra[\Big]{(1+C_P)^2 \|\tilde{b}\|_1^2 + 2c_1 C_P + |\tilde{q}|_1^2}/q_0 + 2$), we obtain
    \begin{align*}
    \frac{1}{2} \frac{\mathrm{d}}{\mathrm{d} t} |w|_1^2 + \frac{\mu}{4} |w|_1^2 \leq \frac{1}{2\mu} |\delta c u|_1^2 + \frac{3}{q_0} \bigl(|\delta b u|_1^2 + \|\nabla \cdot (\delta q \nabla u)\|^2 + \|\delta f\|^2\bigr),
    \end{align*}
    and finally integrating both sides gives
    \begin{align*}
    |w(t)|_1^2 \leq |w(0)|_1^2 e^{-\mu t/2} &+ \frac{1}{\mu} \int_0^t  |\delta c u(s)|_1^2 e^{\mu(s-t)/2} \mathrm{d}s \\
        &+ \frac{3}{q_0} \int_0^t  \bigl(|\delta b u(s)|_1^2 + \|\nabla \cdot (\delta q \nabla u(s))\|^2 + \|\delta f(s)\|^2\bigr) e^{\mu(s-t)/2} \mathrm{d}s.
    \end{align*}
    This ends the proof.  
\end{proof}

\par

Next, we consider the elliptic equations, the one with exact coefficients,
\begin{equation}\label{eq:elliptic-1}
\begin{aligned}
    -\nabla \cdot \bigl( q \nabla u \bigr) + \nabla \cdot (b u) + c u &= f, &&\text{in } \Omega\\
    u &= 0, &&\text{on } \partial \Omega
\end{aligned}
\end{equation}
and the one with perturbations in coefficients,
\begin{equation}\label{eq:elliptic-2}
\begin{aligned}
    -\nabla \cdot \Bigl( (q+\delta q) \nabla v \Bigr) + \nabla \cdot \Bigl((b+\delta b) v\Bigr) + \bigl(c+\delta c\bigr) v &= f + \delta f + \mu I_h(u-v), &&\text{in } \Omega\\
    v &= 0. &&\text{on } \partial \Omega
\end{aligned}
\end{equation}
These two equations admit unique solutions under our assumptions by the standard theory for
linear elliptic equations\cite{evans2010partial}.
The equations( \ref{eq:elliptic-1}) and (\ref{eq:elliptic-2}) can be regarded as steady states of the corresponding parabolic equations (\ref{eq:plm-1}) and (\ref{eq:plm-2}) with the $f$ and $\delta f$ taken to be temporally independent. We thus immediately obtain the following estimates by letting $t\to \infty$ in Lemma \ref{lemma1} and Lemma \ref{lemma2}.

\begin{coro}\label{coro1}
    If $2C_P^{-2} \rbra[\big]{q_0 - \mu C_A^2 h^2 - b_1 C_P} + 2c_0 + \mu\geq \frac{1}{4}\mu$,  we have for $w = u-v$ the error between solutions of (\ref{eq:elliptic-1}) and (\ref{eq:elliptic-2}),
    \begin{equation}\label{eq:coro1-est}
    \lVert w \rVert^2 \leq \frac{16}{\mu^2} \Bigl(\lVert\delta f \rVert^2 + C_P^{-2} \lVert \delta b u \rVert^2  + \lVert \delta c u \rVert^2 + \norm{ \nabla \cdot(\delta q \nabla u) }_{-1}^2\Bigr)
    \end{equation}
\end{coro}

\begin{coro}\label{coro2}
    Suppose that $u, v \in H^2 (\Omega) \cap H^1_0(\Omega)$. If we take $\mu \geq 2 + 4\rbra[\Big]{(1+C_P)^2 \|b+\delta b\|_1^2 + 2c_1 C_P + |q+\delta q|_1^2}/q_0$, and take $h^2 \leq C_A^2q_0/2\mu^2$, we have for $w = u-v$ the error between solutions of (\ref{eq:elliptic-1}) and (\ref{eq:elliptic-2}),
    \begin{equation}\label{eq:coro2-est}
        |w|_1^2 \leq \frac{2}{\mu^2} |\delta c u|_1^2 + \frac{6}{q_0\mu} \Bigl(|\delta b u|_1^2 + \|\nabla \cdot (\delta q \nabla u) \|^2 + \|\delta f\|^2\Bigr).
    \end{equation}
\end{coro}

In addition, since the error $w = u-v$ for (\ref{eq:elliptic-1}) and (\ref{eq:elliptic-2}) satisfies $- \nabla \cdot (\tilde{q} \nabla w) + \nabla \cdot (\tilde{b} w) + \tilde{c} w + \mu I_h w = - \delta f - \nabla \cdot (\delta q \nabla u) + \nabla \cdot (\delta b u) + \delta c u$ with homogeneous Dirichlet condition, an $H^2$ estimate for $w$ follows from the elliptic regularity:
\begin{equation}\label{eq:h2-est}
    \|w\|_2 \leq C \left(\|\delta f\| + \|\nabla \cdot (\delta q \nabla u)\| + |\delta b u|_1 + \|\delta c u\| \right)
\end{equation}
for some $C>0$.

\begin{rmk}\label{rmk2}
The error in the observation $I_h u$ can be included into the error $\delta f$ in effect, by writing $f + \delta f + \mu(I_h u + \delta I_h u - I_h v) = f + (\delta f + \mu\delta I_h u) + \mu I_h(u-v)$ where $\delta I_h u$ is the error in $I_h u$. Thus, Corollary \ref{coro1} and Corollary \ref{coro2}, as well as Lemma \ref{lemma1}, Lemma \ref{lemma2} and (\ref{eq:h2-est}) present also the error estimates for the data error in the observation $I_h u$.
\end{rmk}

\subsection{Lipschitz continuity of the coefficient-to-solution maps}

In the following discussions, we shall set the initial value of (\ref{eq:plm-2}) to be zero for convenience, as the assimilation error caused by the initial value decays exponentially fast to $0$ according to Lemma \ref{lemma1} and Lemma \ref{lemma2}, and we let $f$ be temporal independent as we want to consider steady states. Then, for readability, we rewrite (\ref{eq:plm-2}) as
\begin{equation}\label{eq:coef-to-sol-1}
\begin{aligned}
    v_t -  \nabla \cdot \rbra[\big]{\tilde{q}(x) \nabla v } + \nabla \cdot \bigl( \tilde{b}(x) v \bigr) + \tilde{c}(x) v &= \tilde{f}(x) + \mu I_h (u-v),\quad &&\text{in } (0, T] \times \Omega \\
    v &= 0,  &&\text{on } [0, T] \times \partial \Omega \\
    v(0) &= 0,  &&
\end{aligned}
\end{equation}
and for later use in the derivation of the gradient formulas to recover coefficients $q$ and $f$, we consider here the coefficient-to-solution map
\begin{align*}
    \mathcal{Q} \times L^2(\Omega) &\to H^1(\Omega)\\
    (\tilde{q}, \tilde {f}) &\mapsto v_{\tilde{q}, \tilde{f}}(t),
\end{align*}
where $v_{\tilde{q}, \tilde{f}}$ is the solution of (\ref{eq:coef-to-sol-1}) with coefficient $\tilde{q}$ and source term $\tilde{f}$. When $\tilde{f}$ is fixed in our considerations, we shall denote $v_{\tilde{q}, \tilde{f}}$ by $v_{\tilde{q}}$, and by $v_{\tilde{f}}$ when $\tilde{q}$ is fixed. 

\begin{lem}\label{lem-boundedness-q-to-v}
    Assume the conditions for $h$ and $\mu$ in Lemma \ref{lemma2}. We have $\|v_{\tilde{q}} (t)\|_1^2$ bounded by
    % \begin{align*}
    %     \|A(\tilde{q})\|^2 &\leq 2|u(t)|_1^2 + 2|u(0)|_1^2 e^{-\mu t/2} + \frac{8\|\tilde{q}-q\|^2_{L^\infty(\Omega)}}{\mu} \int_0^t |u(s)|_1^2 e^{C_\mu (s-t)} \mathrm{d}s
    % \end{align*}
    % and
    \begin{align*}
        \|v_{\tilde{q}} (t)\|_1^2 \leq (1+C_P^2)C_t + (1+C_P^2)\frac{3\|\tilde{q}-q\|^2_{1}}{q_0} \int_0^t \|u(s)\|_2^2e^{\mu (s-t)/2} \mathrm{d}s,
    \end{align*}
    where $C_t$ depends on $t$ and is independent of $\tilde{q}$.
\end{lem}

\begin{proof}
    The use of $\|v_{\tilde{q}}(t)\|_1 \leq \|v_{\tilde{q}}(t) - u(t)\|_1 + \|u(t)\|_1$, Lemma \ref{lemma2} and Poincar\'e's inequality gives the result.  
    % \begin{align*}
    %     \|A(\tilde{q})\|^2 \leq 2\|u(t)\|^2 + 2\|w(t)\|^2 &\leq 2\|u(t)\|^2 + 2\|u(0)\|^2 e^{-\mu t/2} + \frac{8}{\mu} \int_0^t \|\nabla\cdot \left( ( \tilde{q} - q ) \nabla u(s)\right)\|_{-1}^2 \mathrm{d}s\\
    %     &\leq  2|u(t)|_1^2 + 2|u(0)|_1^2 e^{-\mu t/2} + \frac{8\|\tilde{q}-q\|^2_{L^\infty(\Omega)}}{\mu} \int_0^t |u(s)|_1^2 e^{C_\mu (s-t)} \mathrm{d}s.
    % \end{align*}
    % and
    % \begin{align*}
    %     |A(\tilde{q})|_1^2 \leq 2|u(t)|_1^2 + 2|w(t)|_1^2 &\leq 2|u(t)|_1^2 + 2|u(0)|_1^2 e^{-\mu t/2} + \frac{6}{q_0} \int_0^t \|\nabla\cdot \left( ( \tilde{q} - q ) \nabla u(s)\right)\|^2 \mathrm{d}s\\
    %     &\leq  2|u(t)|_1^2 + 2|u(0)|_1^2 e^{-\mu t/2} + \frac{6\|\tilde{q}-q\|^2_{1}}{q_0} \int_0^t \|u(s)\|_2^2 e^{\mu (s-t)/2} \mathrm{d}s.
    % \end{align*}
\end{proof}

\begin{lem}\label{lem-continuity-q-to-v}
    Assume the conditions for $h$ and $\mu$ in Lemma \ref{lemma2}. For a fixed $t>0$, the map $\tilde{q} \mapsto v_{\tilde{q}}(t)$ is locally Lipschitz continuous as
    \begin{align}
        \|v_{\tilde{q}_1} (t) - v_{\tilde{q}_2} (t)\|_1 ^2 \leq (1+C_P^2) C_t \frac{\|\tilde{q}_1-\tilde{q}_2\|^2_{1}}{q_0},
    \end{align}
    for some $C_t > 0$ depending on $\tilde{q}_1$, $\tilde{q}_2$, $\mu$ and $t$.
\end{lem}
\begin{proof}
    Denoting $w = v_{\tilde{q}_1} - v_{\tilde{q}_2}$, we have $w_t - \nabla \cdot (\tilde{q}_1 \nabla w) + \nabla \cdot (\tilde{b} w) + \tilde{c} w + \mu I_h w = \nabla \cdot \left( ( \tilde{q}_1 - \tilde{q}_2 )\nabla v_{\tilde{q}_2} \right)$ with $w=0$ on $\partial \Omega$ and $w(0) = 0$. Then following the way in which we proved Lemma \ref{lemma2}, we obtain
    \begin{align*}
        |w(t)|^2_1 \leq \frac{C}{q_0} \int_0^t \|\nabla \cdot (\tilde{q}_1-\tilde{q}_2) \nabla v_{\tilde{q}_2}\|^2 e^{c \mu (s-t)} \mathrm{d}s \leq \frac{C \|\tilde{q}_1-\tilde{q}_2\|_1^2}{q_0} \int_0^t \|v_{\tilde{q}_2}\|_2^2 e^{c \mu (s-t)} \mathrm{d}s.
    \end{align*}
    A similar estimate holds with $v_2$ replaced by $v_1$ as well, so that 
    \begin{equation*}
        |v_{\tilde{q}_1}(t) - v_{\tilde{q}_2}(t)|_1^2\leq \frac{C \|\tilde{q}_1-\tilde{q}_2\|_1^2}{q_0} \min\left\{\int_0^t \|v_{\tilde{q}_1}\|_2^2 e^{c \mu (s-t)} \mathrm{d}s, \int_0^t \|v_{\tilde{q}_2}\|_2^2 e^{c \mu (s-t)} \mathrm{d}s \right\}.    \qedhere
    \end{equation*}
\end{proof}

Similar results can be derived for $v_{\tilde{f}}$ from essentially same discussions.

\begin{lem}\label{lem-boundedness-f-to-v}
    Assume the conditions for $h$ and $\mu$ in Lemma \ref{lemma2}. We have $\|v_{\tilde{f}} (t)\|$ bounded by
    \begin{align*}
        \|v_{\tilde{f}} (t)\|_1^2 \leq (1+C_P^2)C_t + (1+C_P^2)\frac{3}{q_0} \int_0^t \|\tilde{f}(s) - f(s)\|^2 e^{\mu(s-t)/2} \mathrm{d}s,
    \end{align*}
    where $C$ depends on $t$ and is independent of $\tilde{f}$.
\end{lem}

\begin{lem}\label{lem-continuity-f-to-v}
    Assume the conditions for $h$ and $\mu$ in Lemma \ref{lemma2}. For a fixed $t>0$, the map $\tilde{f} \mapsto v_{\tilde{f}}(t)$ is Lipschitz continuous as
    \begin{align*}
        \|v_{\tilde{f}_1} (t) - v_{\tilde{f}_2} (t)\|_1^2  \leq (1+C_P^2) \frac{1 - e^{- \mu t/2}}{q_0 \mu }\|\tilde{f}_1 - \tilde{f}_2\|^2.
    \end{align*}
\end{lem}

\section{Approximated gradients for coefficient reconstruction and the error estimates}\label{sect4}

\subsection{The approximated gradients for the error functional}

In a recent work \cite{newey2025model}, a more mathematically systematic approach is proposed to derive an iterative scheme for recovering constant coefficients in differential equations. It turns out that the feedback term $\mu I_h (u-v)$ provides us with an efficient way to calculate the gradients of some error functional. We now use the techniques in \cite{newey2025model} (which can be easily extended from constant to non-constant cases) to derive a coefficient reconstruction method for non-constant parameters $q(x)$ and also $f(x)$. We shall first consider (\ref{eq:plm-1}) where the source term is independent of $t$, and apply the conclusions to (\ref{eq:elliptic-1}) by viewing it as the steady state of (\ref{eq:plm-1}).

Consider reconstructing the coefficients $q(x)$ and $f(x)$ in (\ref{eq:plm-1}) 
% \begin{equation}
% \begin{aligned}\label{eq:coefid-1}
%     u_t - \nabla \cdot\left(q(x) \nabla u \right) + \nabla \cdot \bigl( b(x) u \bigr) + c(x) u &= f(x), &&\text{in } \Omega_T \\
%     u &= 0,  &&\text{on } \partial \Omega_T \\
%     u(0) &= u_0, &&
% \end{aligned}    
% \end{equation}
utilizing the corresponding data assimilation equation 
\begin{equation}
\begin{aligned}\label{eq:coefid-2}
    v_t - \nabla \cdot \rbra[\big]{\tilde{q}(x) \nabla v } + \nabla \cdot \bigl( b(x) v \bigr) + c(x) v &= \tilde{f}(x) + \mu I_h (u-v), &&\text{in } (0, T] \times \Omega  \\
    v &= 0,  &&\text{on } [0, T] \times \partial \Omega  \\
    v(0) &= 0, &&
\end{aligned}    
\end{equation}
where $b$, $c$ are given exact coefficients and $\tilde{f} = f$ is given when recovering $q$, and $\tilde{q} = q$ is given when recovering $f$. Note that $f(x)$ here is set to be time-independent, while in Lemma \ref{lemma1} and Lemma \ref{lemma2} we considered the more general time-dependent case. 
As before, we denote by $v_{\tilde{q}, \tilde{f}}(t)$ the solution of (\ref{eq:coefid-2}) at time $t$. Fix some $T>0$, we define an error function
\begin{equation}\label{eq:err-functional}
    J(\tilde{q}, \tilde{f}) := \frac{1}{2} \|I_h u(T) - I_h v_{\tilde{q}, \tilde{f}}(T)\|^2.
\end{equation}
We shall provide approximated gradients of $J$ with respect to $q$ and $f$ in the following. This technique avoids numerically solving the adjoint problems in practice.

\subsubsection{An approximated gradient formula for $q$}

Let $\tilde{f}=f$ in (\ref{eq:coefid-2}) be fixed and write $J(\tilde{q}) = J(\tilde{q}, f)$. We have shown in Lemma \ref{lem-continuity-q-to-v} that $\tilde{q} \mapsto v_{\tilde{q}} (T)$ is continuous from $H^1 (\Omega)$ to $H^1_0 (\Omega)$, thus given a small perturbation $\delta q$, we write the solution $v_{\tilde{q}+\delta q}$ to be
\begin{equation}\label{eq:coefid-3}
    v_{\tilde{q}+\delta q} = v_{\tilde{q}} + \delta v_,
\end{equation}
with $\|\delta v\| \leq C \|\delta q\|_{1}$ for some $C$. For simplicity of notations, let us write (\ref{eq:coefid-3}) as $v_{\tilde{q}+\delta q} = v + \delta v$ for some fixed $\tilde{q}$ we are considering. Then it can be computed that for $J(\tilde{q})$ a perturbation of $\delta q$ in $\tilde{q}$ results in
\begin{equation}
\begin{aligned}\label{eq:loss-perturb-q}
    J(\tilde{q} + \delta q) - J(\tilde{q}) = - (I_h \delta v, I_h (u-v))_{L^2} + \frac{1}{2} \|\delta v\|^2= - (\delta v, I_h (u-v))_{L^2} + \frac{1}{2} \|\delta v\|^2
\end{aligned}
\end{equation}
where the last equality follows from that $I_h$ is an $L^2$ projection, and we have $\|\delta v\|^2 = O\left(\|\delta q\|^2_{1} \right)$, so that $- (\delta v, I_h (u-v))_{L^2}$ serves as $(\nabla_q J(\tilde{q}), \delta q)_{L^2}$, and our goal now is to find a practical approximation of $\delta v$. For this, we perturb (\ref{eq:coefid-2}) with $\delta q$ and write
\begin{equation*}
    (v + \delta v)_t - \nabla \cdot \left[(\tilde{q}+\delta q) \nabla (v + \delta v)\right] + \nabla \cdot \left[b (v+\delta v)\right] + c (v + \delta v) = f + \mu I_h \left[u-(v + \delta v)\right],
\end{equation*}
and $\delta v$ satisfies thus
\begin{equation}\label{eq:grad-q-1}
    (\delta v)_t - \nabla \cdot \left[(\tilde{q}+\delta q) \nabla \delta v\right] + \nabla\cdot \left(b \delta v\right) + c \delta v + \mu I_h \delta v = \nabla \cdot (\delta q \nabla v).
\end{equation}
Define an operator $A_\mu v := - \nabla \cdot \left[(\tilde{q}+\delta q) \nabla \delta v\right] + \nabla\cdot \left(b \delta v\right) + c \delta v + \mu I_h \delta v$, we represent the solution of (\ref{eq:grad-q-1}) using Duhamel's principle and make an approximation using the largeness of $\mu$,
\begin{align}\label{eq:grad-q-2}
    \delta v(T) =  e^{-T A\mu}  \delta v(0) + \int_0^T e^{(s-t)A_\mu} \nabla \cdot \left( \delta q \nabla v(s)\right)\mathrm{d} s \sim \frac{1}{\mu} \nabla \cdot \left( \delta q \nabla v(T)\right).
\end{align}
This approximation in (\ref{eq:grad-q-2}) is an operator analogue of Watson's lemma.
Return to (\ref{eq:loss-perturb-q}), we have
\begin{align*}
    J(\tilde{q} + \delta q) - J(q) &= - (I_h \delta v(q), I_h (u-v))_{L^2} + O\left(\|\delta q\|^2_{1}\right)\\
        & \sim - \frac{1}{\mu} ( \nabla \cdot \left(\delta q \nabla v\right), I_h (u-v))_{L^2} + O\left(\|\delta q\|^2_{1}\right)\\
        & = - \frac{1}{\mu} (\delta q, \nabla v \cdot \nabla I_h (u-v))_{L^2} \Bigg\rvert_{t=T} + O\left(\|\delta q\|^2_{1}\right),
\end{align*} 
so that we write, as we wrote $v = v_{\tilde{q}}$,
\begin{equation}\label{eq:grad-q}
    \nabla_{q, \,\text{aprox}} J(\tilde{q}) = \frac{1}{\mu} \nabla v_{\tilde{q}} \cdot \nabla I_h (u-v_{\tilde{q}})\Bigg\rvert_{t=T},
\end{equation}
where $\nabla_{q, \,\text{aprox}} J(\tilde{q})$ means it is an approximation of the gradient $\nabla_{q} J(\tilde{q})$. While we do not intend to rigorously justify (\ref{eq:grad-q-2}) in the present work, our later error analysis of the coefficient reconstruction scheme would show the effectiveness of such an approximation.

Next we show that $\nabla_{q, \,\text{aprox}} J$ is locally Lipschitz continuous from $H^1(\Omega)$ to $L^2 (\Omega)$.

\begin{prop}\label{grad-q-lip}
    There exists $C_{\tilde{q}_1, \tilde{q}_2} > 0$ such that
    \begin{equation}
        \|\nabla_{q, \,\text{aprox}} J(\tilde{q}_1) - \nabla_{q, \,\text{aprox}} J(\tilde{q}_2)\| \leq \frac{C_{\tilde{q}_1, \tilde{q}_2}}{q_0\mu} \|\tilde{q}_1 - \tilde{q}_2\|_{1}.
    \end{equation}
\end{prop}

\begin{proof}
    Direct calculations give for the difference $\nabla_{q, \,\text{aprox}} J(\tilde{q}_1) - \nabla_{q, \,\text{aprox}} J(\tilde{q}_2)$ that, with $u = u(T)$, $v_{\tilde{q}_1} = v_{\tilde{q}_1}(T)$ and $v_{\tilde{q}_2} = v_{\tilde{q}_2} (T)$,
    \begin{align*}
        &\nabla_{q, \,\text{aprox}} J(\tilde{q}_1) - \nabla_{q, \,\text{aprox}} J(\tilde{q}_2) \\
        =\, & \frac{1}{\mu} \nabla v_{\tilde{q}_1} \cdot \nabla I_h \left( v_{\tilde{q}_1} - v_{\tilde{q}_2}\right) + \frac{1}{\mu} \nabla I_h \left(u - v_{\tilde{q}_2}\right)\cdot \nabla \left(v_{\tilde{q}_2} - v_{\tilde{q}_1} \right),
    \end{align*}
    so by Lemma \ref{lem-boundedness-q-to-v} and Lemma \ref{lem-continuity-q-to-v},
    \begin{align*}
        &\|\nabla_{q, \,\text{aprox}} J(\tilde{q}+\delta q) - \nabla_{q, \,\text{aprox}} J(\tilde{q})\| \\
        \leq &\frac{1}{\mu}| v_{\tilde{q}_1}|_1 \cdot |I_h (v_{\tilde{q}_1} - v_{\tilde{q}_2})|_1 + \frac{1}{\mu} | I_h (u - v_{\tilde{q}_2})|_1 \cdot | v_{\tilde{q}_2} - v_{\tilde{q}_1}|_1\\
        \leq & \frac{C}{\mu} \left(1 + \frac{\|\tilde{q}_1-q\|_1}{\sqrt{q_0}}  \right)\frac{\|\tilde{q}_1-\tilde{q}_2\|_1}{\sqrt{q_0}} + \frac{C}{\mu}  \left(|u|_1 + 1 + \frac{\|\tilde{q}_2-q\|_1}{\sqrt{q_0}} \right)\frac{\|\tilde{q}_1-\tilde{q}_2\|_1}{\sqrt{q_0}}\\
        \leq & \frac{C_{\tilde{q}_1, \tilde{q}_2}}{q_0\mu} \|\tilde{q}_1 - \tilde{q}_2\|_{1}. \qedhere
    \end{align*}
\end{proof}

\subsubsection{An approximated gradient formula for $f$}

An approximated gradient can be derived for $J(\tilde{f}) = J(q, \tilde{f})$ in a similar way as for $J(\tilde{q})$. Due to Lemma \ref{lem-continuity-f-to-v}, we may write
\begin{equation*}
    v_{\tilde{f}+\delta f} = v_{\tilde{f}} + \delta v
\end{equation*}
where $\|\delta v\|_1 \leq C \|\delta f\|$ for some $C>0$, and $v_{\tilde{f}} = v_{q, \tilde{f}}$ with the exact $q$ given and fixed. As in (\ref{eq:loss-perturb-q}), we have for a perturbation $\delta f(x) \in L^2(\Omega)$,
\begin{equation}
\begin{aligned}\label{eq:loss-perturb-f}
    J(\tilde{f} + \delta f) - J(\tilde{f}) &= - (\delta v, I_h (u-v))_{L^2} + \frac{1}{2} \|\delta v\|^2,
\end{aligned}
\end{equation}
and to find an approximated $\delta v$, we solve
\begin{equation*}
    (\delta v)_t - \nabla \cdot \left(q \nabla \delta v\right) + \nabla\cdot \left(b \delta v\right) + c \delta v + \mu I_h \delta v = \delta f.
\end{equation*}
Then for large $\mu$ and $T$, we have
\begin{align*}
    \delta v(T) = e^{-T A\mu}  \delta v(0) + \int_0^T e^{(s-t)A_\mu} \delta f \mathrm{d} s \sim \frac{1}{\mu} \delta f.
\end{align*}
Thus,
\begin{align*}
    J(\tilde{f} + \delta f) - J(f) \sim - \frac{1}{\mu} (\delta f, I_h (u-v))_{L^2} \Bigg\rvert_{t=T} + O\left(\|\delta f\|^2\right),
\end{align*} 
and we write $\nabla_{f, \,\text{aprox}} J(\tilde{f}) = -\frac{1}{\mu} I_h (u-v) \Big\rvert_{t=T}$. This $\nabla_{f, \,\text{aprox}} J$ is Lipschitz continuous.

\begin{prop}\label{grad-f-lip}
    There exists $C > 0$ such that
    \begin{equation*}
        \|\nabla_{f, \,\text{aprox}} J (\tilde{f}_1) - \nabla_{f, \,\text{aprox}} J(\tilde{f}_2)\| \leq \frac{C}{\mu^{3/2}} \|\tilde{f}_1 - \tilde{f}_2\|.
    \end{equation*}
\end{prop}
\begin{proof}
    Note that $\nabla_{f, \,\text{aprox}} J (\tilde{f}_1) - \nabla_{f, \,\text{aprox}} J (\tilde{f}_2) =\frac{1}{\mu} I_h (v_{\tilde{f}_1}-v_{\tilde{f}_2}) \Big\rvert_{t=T}$, then Lemma \ref{lem-continuity-f-to-v} implies the conclusion.  
\end{proof}

\subsubsection{Application to the elliptic case} \label{sec:ellip-aprox-grad}

Now consider the reconstruction of $q$ and $f$ in the elliptic equation (\ref{eq:elliptic-1}) regarded as the steady state of (\ref{eq:plm-1}) using the assimilation equation 
\begin{equation}\label{eq:elliptic-3}
\begin{aligned}
    -\nabla \cdot ( \tilde{q} \nabla v ) + \nabla \cdot (b v) + c v &= f + \mu I_h(u-v), &&\text{in } \Omega\\
    v &= 0, &&\text{on } \partial \Omega
\end{aligned}
\end{equation}
The error functional is (\ref{eq:err-functional}) interpreted with $T\to\infty$. We use the finite element method(FEM) to recover $q$ and the suppose reconstruction $q_{h_q}$ is in a finite element space $V_{h_q}$ with ${\Phi_j}_{1\leq j \leq m_q}$ its basis. Then we can write $\tilde{q}_{h_q} = \sum_{1\leq j \leq m_q} d_j \Phi_j$, and with (\ref{eq:grad-q}) we update each $d_j$ iteratively, in a typical gradient descent formulation for example, by 
\begin{equation}\label{eq:coef-q-update}
    d_j^{(n+1)} - d_j^{(n)} = -\frac{r}{\mu} \int_\Omega  \nabla v_{\tilde{q}_{h_q}^{(n)}} \cdot \nabla I_h (u-v_{\tilde{q}_{h_q}^{(n)}}) \cdot \Phi_j,
\end{equation}
where $r$ is a step size to be taken appropriately. The similar is for the reconstruction of $f$, i.e., if $\tilde{f}_{h_f} = \sum_{1\leq j \leq m_f}e_j \Phi_j$ in a finite element space $V_{h_f}$, then 
\begin{equation}\label{eq:coef-f-update}
    e_j^{(n+1)} - e_j^{(n)} = \frac{r}{\mu} \int_\Omega I_h (u-v_{\tilde{q}_{h_q}^{(n)}}) \cdot \Phi_j,
\end{equation}
as a simple sample of gradient descent. Other gradient methods can be applied as well, and Proposition \ref{grad-q-lip} and \ref{grad-f-lip} carry over to the elliptic case with (see also Lemma \ref{lem-continuity-q-to-v} and Lemma \ref{lem-continuity-f-to-v}, letting $t\to \infty$ there)
\begin{equation}\label{eq:q-elliptic-lipbound}
    \|\nabla_{q, \,\text{aprox}} J(\tilde{q}_1) - \nabla_{q, \,\text{aprox}} J(\tilde{q}_2)\| \leq \frac{C_{\tilde{q}_1, \tilde{q}_2}}{q_0 \mu^{3/2}} \|\tilde{q}_1 - \tilde{q}_2\|_{1},
\end{equation}
\begin{equation}\label{eq:f-elliptic-lipbound}
    \|\nabla_{f, \,\text{aprox}} J (\tilde{f}_1) - \nabla_{f, \,\text{aprox}} J(\tilde{f}_2)\| \leq \frac{C}{\mu^{3/2}} \|\tilde{f}_1 - \tilde{f}_2\|.
\end{equation}
which provides good properties for the application of gradient methods.

\subsection{Error estimates for reconstruction of $q$}

In view of formula (\ref{eq:grad-q}), we write our formulation for the reconstruction of the parameter $q$ in the following form, which is more abstract than (\ref{eq:coef-q-update}): Introduce a parameter $z\in [0, \infty)$ (one may imagine $z$ as a time variable) and let $\tilde{q}_{h_q} \in V_{h_q}$ evolve with
% \begin{equation}\label{eq:idetification-q-cont}
%     \begin{cases}
%         - \nabla \cdot (\tilde{q}(z) \nabla v(z)) + \nabla \cdot (b v(z)) + c v(z) = f + \mu I_h (u-v(z)),\\
%         \frac{\mathrm{d}}{\mathrm{d} z} \tilde{q}(z) = - \frac{1}{\mu} \nabla v \cdot \nabla  I_h (u-v), \quad \tilde{q}(0) \text{ the given initial guess},
%     \end{cases}
% \end{equation}
% where $v(z)$ is the solution of the first equation in (\ref{eq:idetification-q-cont}) with homogeneous Dirichlet boundary condition.  Let $(\Phi_j)_{1\leq j \leq m}$ be a basis of the finite dimensional space $V_{h_q}$, the spatially discrete version is, for $\tilde{q}_h \in V_{h_q}$,
\begin{equation}\label{eq:idetification-q-discrete}
    \begin{cases}
        - \nabla \cdot (\tilde{q}_{h_q} (z) \nabla v(z)) + \nabla \cdot (\tilde{b} v(z)) + \tilde{c} v(z) = \tilde{f} + \mu I_h (u-v(z)),\\
        \int_\Omega \Phi_j \frac{\mathrm{d}}{\mathrm{d} z} \tilde{q}_{h_q} (z) = - \frac{1}{\mu} \int_\Omega \Phi_j \nabla v \cdot \nabla  I_h (u-v(z)), \quad \tilde{q}_{h_q} (0) \in V_{h_q} \text{ the given initial guess}.
    \end{cases}
\end{equation}
Here $\int_\Omega \Phi_j \frac{\mathrm{d}}{\mathrm{d} z} \tilde{q}_h (z)$ should be interpreted as $\frac{\mathrm{d}}{\mathrm{d} z} \tilde{q}_{h_q} (z)$ in the direction $\Phi_j$, and thus $\int_\Omega \varphi \frac{\mathrm{d}}{\mathrm{d} z} \tilde{q}_{h_q} (z)$ is not defined for $\varphi \notin V_{h_q}$. We emphasize that the inputted $\tilde{b}$, $\tilde{c}$ and $\tilde{f}$ are not necessarily the exact coefficients, and the noise in the data $I_h u$ can be regarded as error in $\tilde{f}$ (Remark \ref{rmk2}).
We have from (\ref{eq:q-elliptic-lipbound}) and the inverse estimate (\ref{eq:inverse-estimate}) that $\nabla v(z) \cdot \nabla  I_h (u-v(z)) = \nabla v_{\tilde{q}_{h_q}} \cdot \nabla  I_h (u-v_{\tilde{q}_{h_q}})$ is locally Lipschitz continuous with fixed $h_q$ for $\tilde{q}_{h_q}$ in $L^2(\Omega)$. Hence, (\ref{eq:idetification-q-discrete}) has a unique solution in $L^2 (\Omega)$ for $z\in [0,\infty)$ \cite{neuberger2009sobolev}.

% In our formulations, $\tilde{q}$ is regarded as a function evolving with $z$ in the continuous version, and the discrete version (\ref{eq:idetification-q-discrete}) is regarded as a finite difference discretization of (\ref{eq:idetification-q-cont}). 

We shall focus on the semi-discrete (\ref{eq:idetification-q-discrete}) for the error analysis. First, we need the following lemma.

\begin{lem}\label{lem6}
    Suppose $q_{h_q}(z) \in V_{h_q} (z\geq 0)$ satisfies (\ref{eq:idetification-q-discrete}), and $I_{h_q} q$ is the $L^2$ projection of $q$ in $V_{h_q}$ with $I_{h_q} q (x) \geq C_{I_{h_q}} q_0$ ($\forall x \in \Omega$) for some $C_{I_{h_q}}>0$. Then, with $w(z) = u-v(z)$ and $\delta q(z) = \tilde{q}_{h_q}(z) - q$, we have 
    \begin{multline*}
        \mu \frac{\mathrm{d}}{\mathrm{d}z}\|\delta q(z)\|^2 + C_{I_{h_q}} q_0 |w|_1^2 + \mu \|I_h w\|^2 \leq C \|w\|\left( \|\nabla \cdot \left((q-I_{h_q} q) \nabla u\right)\| + |\delta b u|_1 + \|\delta c u\| + \|\delta f\|\right) \\
        + C \|w\||w|_1 + C h |w|_1 \left(\|\delta f\| + \|\nabla \cdot (\delta q \nabla u)\| + |\delta b u|_1 + \| \delta c u \|\right).
    \end{multline*}
    where $\delta b = \tilde{b} - b$, $\delta c = \tilde{c} - c$ and $\delta f = \tilde{f} - f$, and $C > 0$ depends on $C_A$, $q_1$ and $b_1$.
    % In particular, we have $\frac{\mathrm{d}}{\mathrm{d}z} \|\delta q(z)\|^2 \leq 0$ so the error $\|\delta q(z)\|$ is monotone decreasing if $\delta b$, $\delta c$ and $\delta f$ are $0$ (the given coefficients $b$, $c$ and $f$ are exact).
\end{lem}

\begin{proof}
    Denote $\delta I_{h_q} q = \tilde{q}_{h_q} - I_{h_q} q$. It can be calculated that for $w$ we have
    \begin{equation*}
        -\nabla \cdot (I_{h_q} q \nabla w) + \nabla \cdot (\delta I_{h_q} q \nabla v) + \nabla \cdot (\tilde{b} w) + \tilde{c} w + \mu I_h w = \nabla \cdot ((q-I_{h_q} q) \nabla u) + \nabla \cdot \delta b u + \delta c u - \delta f,
    \end{equation*}
    and thus inserting $I_h w$ into its weak formulation, we obtain
    \begin{multline}\label{eq:lem4-1}
        \int_\Omega I_{h_q} q \nabla w \cdot \nabla I_hw - \delta I_{h_q} q \nabla v \cdot \nabla I_h w +(b w) \cdot \nabla I_h w + c I_h w^2 + \mu I_h w^2\\
        = \int_\Omega \left(\nabla \cdot ((q-I_{h_q} q) \nabla u) + \nabla \cdot \delta b u + \delta c u - \delta f \right) I_h w.
    \end{multline}
    On the other hand, since $\int_\Omega \Phi_j \frac{\mathrm{d}}{\mathrm{d} z} \tilde{q}_{h_q}(z) = - \frac{1}{\mu} \int_\Omega \Phi_j \nabla v \cdot \nabla  I_h (u-v)$, replacing $\Phi_j$ with $\delta I_{h_q} q \in V_{h_q}$ gives
    \begin{equation}\label{eq:lem4-2}
        \int_\Omega \delta I_{h_q} q \frac{\mathrm{d}}{\mathrm{d} z} \tilde{q}_{h_q}(z) + \frac{1}{\mu} \delta I_{h_q} q \nabla v \cdot \nabla  I_h (u-v) = 0.
    \end{equation}
    Then adding (\ref{eq:lem4-1}) to (\ref{eq:lem4-2}) and noting that $\int_\Omega \delta I_{h_q} q \frac{\mathrm{d}}{\mathrm{d} z} \tilde{q}_{h_q}(z) = \frac{1}{2}\frac{\mathrm{d}}{\mathrm{d}z}\|\delta I_{h_q} q (z)\|^2 = \frac{1}{2}\frac{\mathrm{d}}{\mathrm{d}z}\|\delta q (z)\|^2$ yields
    \begin{multline}\label{eq:lem4-3}
        \frac{\mu}{2}\frac{\mathrm{d}}{\mathrm{d}z}\|\delta q(z)\|^2 + \int_\Omega I_{h_q} q \nabla w \cdot \nabla I_hw - (b w) \cdot \nabla I_h w + c I_h w^2 + \mu I_h w^2\\
        = \int_\Omega \left(\nabla \cdot ((q-I_{h_q} q) \nabla u) + \nabla \cdot \delta b u + \delta c u - \delta f \right) I_h w.
    \end{multline}
    % Since $\int_\Omega I_{h_q} q\nabla w \cdot \nabla I_h w - (b w) \cdot \nabla I_h w \leq \|I_{h_q}\|_{L^\infty} C_A |w|_1^2 + b_1 C_A \|w\| |w|_1 \leq C_{b,q} \|w\|_1^2$ and $\int_\Omega c I_h w^2 \geq 0$, from (\ref{eq:lem4-3}) we obtain
    Next, since $\int_\Omega (b w) \cdot \nabla I_h w \leq b_1 C \|w\| |w|_1$, and using (\ref{eq:h2-est}) we have
    \begin{align*}
        \int_\Omega I_{h_q} q\nabla w \cdot \nabla I_h w \geq C_{I_{h_q}} q_0 |w|_1^2 - C q_1 |w|_1 |w- I_h w|_1 \geq & C_{I_{h_q}} q_0 |w|_1^2 - C q_1 h |w|_1 |w|_2\\
        \geq & C_{I_{h_q}} q_0 |w|_1^2 - C q_1 h |w|_1 \left(\|\delta f\| + \|\nabla \cdot (\delta q \nabla u)\| + |\delta b u|_1 + \| \delta c u \|\right).
    \end{align*}
    Thus, in (\ref{eq:lem4-3}) we have for some $C > 0$ depending on $C_A$, $q_1$ and $b_1$  that
    \begin{multline*}
        \mu \frac{\mathrm{d}}{\mathrm{d}z}\|\delta q(z)\|^2 + C_{I_{h_q}} q_0 |w|_1^2 + \mu \|I_h w\|^2 \leq C \|w\|\left( \|\nabla \cdot \left((q-I_{h_q} q) \nabla u\right)\| + |\delta b u|_1 + \|\delta c u\| + \|\delta f\|\right) \\
        + C \|w\||w|_1 + C h |w|_1 \left(\|\delta f\| + \|\nabla \cdot (\delta q \nabla u)\| + |\delta b u|_1 + \| \delta c u \|\right). \qedhere
    \end{multline*}
    
    % It can be now seen that we have $\mu \frac{\mathrm{d}}{\mathrm{d}z}\|\delta q(z)\|^2 + C(q_0|w|_1^2 + \mu \|I_h w\|^2) \leq \int_\Omega (\nabla \cdot ((q-I_{h_q} q) \nabla u) + \nabla \cdot \delta b u + \delta c u - \delta f ) I_h w$ if $\int_\Omega I_{h_q} q \nabla w \cdot \nabla I_h w - (b w) \cdot \nabla I_h w + \frac{\mu}{2} I_h w^2 \geq C q_0 |w|_1^2$ for some positive $C$. To see what ensures this, we estimate that
    % \begin{align*}
    %     \int_\Omega I_{h_q} q \nabla w \cdot \nabla I_h w &\geq q_0 |w|_1^2 - \|\nabla \cdot (I_{h_q} q \nabla w)\|_{-1} \|I_h w - w\|\\
    %     &\geq q_0 |w|_1^2 - q_1 C_A C_P h |w|_1^2\\
    %     &= (q_0 - q_1 C_A C_P h) |w|_1^2
    % \end{align*}
    % and by the inverse estimate, 
    % \begin{align*}
    %     \int_\Omega (b w) \cdot \nabla I_h w \leq b_1 \|w\| |I_h w|_1 \leq b_1 C_A C_{\text{inv}} h^{-1} \|w\|^2.
    % \end{align*}
    % Hence if $q_0 - q_1 C_A C_P h > q_0/2$ and $b_1 C_A C_{\text{inv}} h^{-1} \leq \mu/2$, we have $\int_\Omega I_{h_q} q \nabla w \cdot \nabla I_h w -(b w) \cdot \nabla I_h w + \frac{\mu}{2} I_h w^2 \geq \frac{q_0}{2} |w|_1^2$ and this gives our conclusion.  
\end{proof}

\begin{rmk}
    \cite{hoffmann1985identification} inspires the preceding Lemma \ref{lem6}. In fact, if we use $\mu(u-v)$ instead of $\mu I_h (u-v)$ as the feedback term and replace $\tilde{q}_{h_q}$ by $\tilde{q} \in H^1(\Omega)$ in (\ref{eq:idetification-q-discrete}), we may show that $\frac{\mathrm{d}}{\mathrm{d}z}\|\tilde{q}(z)-q\|^2 < 0$ so $\|\tilde{q}(z)-q\|^2$ is monotonic-decreasing, which is a situation similar to \cite{hoffmann1985identification}. Lemma \ref{lem6} is a modification of this result in the finite-dimensional case and takes into account the error in other coefficients.
\end{rmk}

Next, we derive an ODE inequality for the $L^2$ error $\|\delta q(z)\|$. We shall assume the following \textit{positivity condition}\cite{bonito2017diffusion}: For some $\beta \geq 0$, it holds that
\begin{equation}\label{eq:positivity-condition}\tag{$\text{PC}$}
    \left(|\nabla u|^2 - \frac{1}{4}\Delta (u^2)\right)(x) \geq c_{\text{pc}} \text{dist}(x, \partial \Omega)^\beta, \quad \text{for any } x\in\Omega.
\end{equation}
% This condition can be directly verified by the problem data since it depends only on the solution $u$ without the knowledge of $q$ and $f$.

\begin{rmk}
    The positivity condition (\ref{eq:positivity-condition}) is different from the one that appears in \cite{bonito2017diffusion}, which is $\left(q |\nabla u|^2 + f u\right)(x) \geq c \,\text{dist}(x, \partial \Omega)^\beta$. For our positivity condition (\ref{eq:positivity-condition}), we emphasize that it can be verified directly from the solution $u$ itself without knowledge of $q$ and $f$. 
\end{rmk}

\begin{lem}\label{lem-deltaq-ineq}
    Assume (\ref{eq:positivity-condition}) for some $\beta \geq 0$ and the conditions of Lemma \ref{lemma1}, Lemma \ref{lemma2} and Lemma \ref{lem6}, and $(8C_A^2 C_{\text{inv}}^2 h_q^{-2} q_1^2 + b_1^2 + C_P^2 c_1^2)/\mu \leq  C_{I_{h_q}} q_0$ with $C_{I_{h_q}}$ in Lemma \ref{lem6}. Let $\tilde{q}_{h_q}$ be in the finite dimensional space $V_{h_q}$ and $q \in H^m(\Omega)$ ($m\geq 1$), we have for $\delta q = \tilde{q}_{h_q} - q$,
    \begin{align}\label{eq:lem-deltaq-ineq}
    \begin{split}
       &C_{q_1,\Omega,\text{pc}}^{-1} \mu^2 |u|_1^2 \frac{\mathrm{d}}{\mathrm{d}z}\|\delta q\|^2 + C_{q_1,\Omega,\text{pc}}\|\delta q\|^{2(1+\beta)} - C_{q_1,\Omega,\text{pc}}\left( \frac{1}{2} +  \frac{3}{\sqrt{\mu}} + h\sqrt{\mu} \right) \|\delta q\|^2\\
        \leq \,& M h_q^{-2} (1 + h \sqrt{\mu} + \frac{1}{\sqrt{\mu}}) \left(|\delta b u|_1^2 + \|\delta c u\|^2 + \|\delta f\|^2 \right) + M \left( 1 + h_q^2\left(1+h\sqrt{\mu} + \frac{1}{\sqrt{\mu}}\right) \right) h_q^{2(m-2)} |q|_m^2,
    \end{split}
    \end{align}
    where $M$ depends on $|u|_1$, $\|u\|_2$, $C_A$, $C_{\textbf{inv}}$ and $C_{q_1,\Omega,\text{pc}}$
\end{lem}

\begin{proof}
    For $w(z) = u - v(z)$, we have
    \begin{equation*}
        -\nabla \cdot (\delta q \nabla u) = -\nabla\cdot (\tilde{q}_{h_q} \nabla w) + \nabla \cdot (b w) + \mu I_h w + c w + \nabla \cdot \delta b u + \delta c u - \delta f,
    \end{equation*}
    and thus for any $\varphi \in H^1_0 (\Omega)$,
    \begin{equation}\label{eq:lem-7-0}
        \int_\Omega \delta q \nabla u\cdot \nabla\varphi = \int_\Omega \tilde{q}_{h_q} \nabla w\cdot \nabla \varphi + (b w)\cdot \nabla \varphi + c w \varphi + \mu I_h w \varphi + (\nabla \cdot \delta b u + \delta c u - \delta f)\varphi.
    \end{equation}
    For the left side of (\ref{eq:lem-7-0}), the choice $\varphi = \delta q u$ gives (which is also the reason for this choice)
    \begin{align*}
         \int_\Omega \delta q \nabla u\cdot \nabla(\delta q u) &= \int_\Omega \delta q^2 |\nabla u|^2 + \delta qu \nabla \delta q\cdot \nabla u\\
         & = \int_\Omega \delta q^2 |\nabla u|^2 + \frac{1}{4}\nabla (\delta q^2) \cdot \nabla (u^2)\\
         & = \int_\Omega \delta q^2 |\nabla u|^2 - \frac{1}{4}\delta q^2 \Delta (u^2) = \int_\Omega \delta q^2 \left(|\nabla u|^2 - \frac{1}{4}\Delta (u^2)\right).
    \end{align*}
    Thus with the positivity condition (\ref{eq:positivity-condition}), denoting $\Omega_\rho = \{x\in\Omega : \, \text{dist}(x, \partial \Omega) > \rho \}$ with some $\rho$ to be determined , we may estimate (see also \cite{bonito2017diffusion}, \cite{jin2021error}, \cite{cen2024numerical} for such derivation)
    \begin{align*}
        \int_{\Omega_\rho} \delta q^2 = \rho^{-\beta} \int_{\Omega_\rho} \delta q^2 \rho^\beta \leq  \rho^{-\beta} c_{\text{pc}}^{-1} \int_{\Omega_\rho} \delta q^2  \left(|\nabla u|^2 - \frac{1}{4}\Delta (u^2)\right) \leq  \rho^{-\beta} c_{\text{pc}}^{-1} \int_{\Omega} \delta q^2  \left(|\nabla u|^2 - \frac{1}{4}\Delta (u^2)\right),
    \end{align*}
    and 
    \begin{align*}
        \int_{\Omega - \Omega_\rho} \delta q^2 \leq \rho C_{q_1,\Omega} .
    \end{align*}
    Hence setting $\rho = \left[\int_{\Omega} \delta q^2  \left(|\nabla u|^2 - \frac{1}{4}\Delta (u^2)\right)\right]^{1+\beta}$ we see that for some $C_{q_1,\Omega,\text{pc}}>0$,
    \begin{equation*}
       C_{q_1,\Omega,\text{pc}} \int_{\Omega} \delta q^2 \leq \left[\int_{\Omega} \delta q^2  \left(|\nabla u|^2 - \frac{1}{4}\Delta (u^2)\right)\right]^{\frac{1}{1+\beta}},
    \end{equation*}
    which is $C_{q_1,\Omega,\text{pc}}\|\delta q\|^{2(1+\beta)} \leq \int_{\Omega} \delta q^2  \left(|\nabla u|^2 - \frac{1}{4}\Delta (u^2)\right)$.
    
    Next, inserting $\varphi = \delta q u \in H^1_0 (\Omega)$ in the right side of (\ref{eq:lem-7-0}), we have for some $C, \overline{C} > 0$ to be determined later that,
    \begin{align}
       &\int_\Omega \tilde{q}_{h_q} \nabla w\cdot \nabla \delta q u + (\tilde{b} w)\cdot \nabla \delta q u + \tilde{c} w \delta q u + \mu I_h w \delta q u  + (\nabla \cdot \delta b u + \delta c u - \delta f)\delta q u\nonumber\\
       \leq \, & q_1 |w(z)|_1 |\delta q(z) u|_1 + b_1 |w(z)|_1 \|\,\delta q(z) u\| + c_1 \|w(z)\|\,\|\delta q(z) u\| + \mu \|I_h w(z)\|\, \|\delta q(z) u\| + \int_\Omega (\nabla \cdot \delta b u + \delta c u - \delta f)\delta q u \nonumber\\
       \leq &\, C^{-1} \mu \left( \frac{C \overline{C}^{-1} h_q^{-2} q_1^2 + b_1^2 + C_P^2 c_1^2}{\sqrt{\mu}} |u|_1^2  |w|_1^2 + \mu \|u\|^2 \|I_h w\|^2  \right)  + C^{-1} \left(|\delta b u|_1^2 + \|\delta c u\|^2 + \|\delta f\|^2 \right) \|u\|^2 \nonumber\\
       &\quad + \left(\frac{C}{\sqrt{\mu}} + \frac{C}{4}\right) \|\delta q\|^2 + h_q^2 \frac{\overline{C}}{\sqrt{\mu}} |\delta q|_1^2 \nonumber\\ %%%%%%
       \leq &\, -C^{-1}\mu^2 |u|_1^2 \frac{\mathrm{d}}{\mathrm{d}z}\|\delta q\|^2 + C^{-1} \mu |u|_1^2  \left( \|\nabla \cdot \left((q-I_{h_q} q) \nabla u\right)\| + |\delta b u|_1 + \|\delta c u\| + \|\delta f\|\right) \|w\|\nonumber\\
       &  + C^{-1} \mu |u|_1^2 \|w\| |w|_1 + C^{-1} \mu h |u|_1^2 |w|_1 (\|\delta f\| + \|\nabla \cdot (\delta q \nabla u) \| + |\delta b u|_1 + \|\delta c u\|)\label{eq:lem-7-1.1}\\
       & + C^{-1}\left(|\delta b u|_1^2 + \|\delta c u\|^2 + \|\delta f\|^2 \right) \|u\|^2 + \left(\frac{C}{\sqrt{\mu}} + \frac{C}{4}\right) \|\delta q\|^2 + h_q^2 \frac{\overline{C}}{\sqrt{\mu}} |\delta q|_1^2 \nonumber \\
       \leq &\, -C^{-1}\mu^2 |u|_1^2 \frac{\mathrm{d}}{\mathrm{d}z}\|\delta q\|^2 + C^{-1} \tilde{C}^{-2} h_q^{-2} |u|_1^4 (1 + h \sqrt{\mu} + \frac{1}{\sqrt{\mu}}) \left(|\delta b u|_1^2 + \|\delta c u\|^2 + \|\delta f\|^2 \right)  \label{eq:lem-7-1.2}\\
       & + C^{-1} \tilde{C}^{-2} h_q^{-2} |u|_1^4 \|u\|_2^2 \|q - I_{h_q}q\|_1^2 + \left[ \tilde{C} h_q^2 \|u\|_2^2 \left(h \sqrt{\mu} + \frac{1}{\sqrt{\mu}} + 1\right) + h_q^{2} \frac{\overline{C}}{\sqrt{\mu}} \right] \|\delta q\|_1^2 + \left(\frac{C}{\sqrt{\mu}} + \frac{C}{4}\right) \|\delta q\|^2\nonumber,
    \end{align}
    where we used Lemma \ref{lem6} in (\ref{eq:lem-7-1.1}) with $(C \overline{C}^{-1}h_q^{-2} q_1^2 + b_1^2 + C_P^2 c_1^2)/\mu \leq  C_{I_{h_q}} q_0$, and Corollary \ref{coro1} and Corollary \ref{coro2} in (\ref{eq:lem-7-1.2}). Then, using the following estimate for $\|\delta q\|_1$,
    \begin{align*}
        \|\delta q\|_1^2 = \|\tilde{q}_{h_q} - I_{h_q} q + I_{h_q} q - q\|_1^2 &\leq 2 \|\tilde{q}_{h_q} - I_{h_q} q\|_1^2 + 2\|I_{h_q} q - q\|_1^2 \leq 2 C_A^2C_{\text{inv}}^2 h_q^{-2} \|\tilde{q}_{h_q} - q\|^2 + 2\|I_{h_q} q - q\|_1^2,
    \end{align*}
    and choosing $C = C_{q_1,\Omega,\text{pc}}$, $\overline{C} = C_A^{-2}C_{\text{inv}}^{-2} C_{q_1,\Omega,\text{pc}}/8$ and $\tilde{C} = C_A^{-2}C_{\text{inv}}^{-2} \|u\|_2^{-2}C_{q_1,\Omega,\text{pc}}/8$, we continue to estimate to obtain
    \begin{align}
    &\int_\Omega \tilde{q}_{h_q} \nabla w\cdot \nabla \delta q u + (\tilde{b} w)\cdot \nabla \delta q u + \tilde{c} w \delta q u + \mu I_h w \delta q u  + (\nabla \cdot \delta b u + \delta c u - \delta f)\delta q u\nonumber\\
       \leq &\, -C_{q_1,\Omega,\text{pc}}^{-1} \mu^2 |u|_1^2 \frac{\mathrm{d}}{\mathrm{d}z}\|\delta q\|^2 + C_{q_1,\Omega,\text{pc}}\left( \frac{1}{2} +  \frac{3}{\sqrt{\mu}} + h\sqrt{\mu} \right) \|\delta q\|^2\nonumber\\
       & + M h_q^{-2} (1 + h \sqrt{\mu} + \frac{1}{\sqrt{\mu}}) \left(|\delta b u|_1^2 + \|\delta c u\|^2 + \|\delta f\|^2 \right) + M \left( 1 + h_q^2\left(1+h\sqrt{\mu} + \frac{1}{\sqrt{\mu}}\right) \right) h_q^{-2} \|q - I_{h_q}q\|_1^2 \nonumber\\
       \leq &\, -C_{q_1,\Omega,\text{pc}}^{-1} \mu^2 |u|_1^2 \frac{\mathrm{d}}{\mathrm{d}z}\|\delta q\|^2 + C_{q_1,\Omega,\text{pc}}\left( \frac{1}{2} +  \frac{3}{\sqrt{\mu}} + h\sqrt{\mu} \right) \|\delta q\|^2 \nonumber\\
       & + M h_q^{-2} (1 + h \sqrt{\mu} + \frac{1}{\sqrt{\mu}}) \left(|\delta b u|_1^2 + \|\delta c u\|^2 + \|\delta f\|^2 \right) + M \left( 1 + h_q^2\left(1+h\sqrt{\mu} + \frac{1}{\sqrt{\mu}}\right) \right) h_q^{2(m-2)} |q|_m^2 \nonumber,
    \end{align}
    with some $M$ depending on $|u|_1$, $\|u\|_2$, $C_A$, $C_{\textbf{inv}}$ and $C_{q_1,\Omega,\text{pc}}$.
    
    Combining the estimates for the two sides of (\ref{eq:lem-7-0}) , we see
    \begin{align*}
        &C_{q_1,\Omega,\text{pc}}^{-1} \mu^2 |u|_1^2 \frac{\mathrm{d}}{\mathrm{d}z}\|\delta q\|^2 + C_{q_1,\Omega,\text{pc}}\|\delta q\|^{2(1+\beta)} - C_{q_1,\Omega,\text{pc}}\left( \frac{1}{2} +  \frac{3}{\sqrt{\mu}} + h\sqrt{\mu} \right) \|\delta q\|^2\\
        \leq \,& M h_q^{-2} \left(1 + h \sqrt{\mu} + \frac{1}{\sqrt{\mu}}\right) \left(|\delta b u|_1^2 + \|\delta c u\|^2 + \|\delta f\|^2 \right) + M \left( 1 + h_q^2\left(1+h\sqrt{\mu} + \frac{1}{\sqrt{\mu}}\right) \right) h_q^{2(m-2)} |q|_m^2,
    \end{align*}
    This completes the proof. 
\end{proof}
    % Thus, we see finally that if $\mu$ and $h_q$ are such that  $\frac{C_{q_1,\Omega,\text{pc}}}{2} - \frac{c_1 \|u\|_2^2}{\mu} - M_{u,\mu} C_A^2 C_{\text{inv}}^2 h_q^{-2}>0$, we have from the last inequality that
    % \begin{equation*}
    %     \lim_{z\to\infty} \|\delta q(z)\|^2 \leq M_{u, \mu}C_A^2 h_q^{2(m-1)}|q|_m^2 + c_2\left( \frac{1}{\mu} + \frac{\|u\|^2}{C_{q_1,\Omega,\text{pc}}}\right) (|\delta b u|_1^2 + \|\delta c u\|^2 + \|\delta f\|^2).
    % \end{equation*}
    
    % Then integrating for $z$ on $(z_1, z_2)$ and use that $\|\delta q(z)\|^2$ is non-increasing, we have
    % \begin{equation*}
    %     C_{q_1,\Omega,\text{pc}} \int_{z_1}^{z_2} \|\delta q(z)\|^{\frac{2}{1+\beta}} \mathrm{d}z \leq \int_{z_1}^{z_2} \int_{\Omega} \delta q(z)^2  \left(|\nabla u|^2 - \frac{1}{4}\Delta (u^2)\right) \mathrm{d}x\mathrm{d}z
    % \end{equation*}
    
    % Return finally to (\ref{eq:lem-7-0}), integrating on $(z_1, z_2)$, we conclude
    % \begin{equation*}
    %     C_1 \mu^2 \sqrt{\mu} \|\delta q(z_2)\|^2 + C_{q_1,\Omega,\text{pc}} \int_{z_1}^{z_2} \|\delta q(z)\|^{\frac{2}{1+\beta}} \mathrm{d}z \leq C_1 \mu^2 \sqrt{\mu} \|\delta q(z_1)\|^2 +  \frac{C_2}{\sqrt{\mu}}\int_{z_1}^{z_2} \|\nabla \cdot(\delta q(z)\nabla u)\|^2 + |\delta q(z) u|_1^2.
    % \end{equation*}    

Let us investigate the estimate in Lemma \ref{lem-deltaq-ineq} to obtain $\lim_{z\to\infty} \|\delta q(z)\|^2$. For ease of notation, we write (\ref{eq:lem-deltaq-ineq}) as an ODE inequality for some $\theta(z)$
\begin{equation}\label{eq:ode-1}
    \theta'(z) + C_1 \theta(z)^{1+\beta} - C_2 \theta (z) \leq C,
\end{equation}
where $C_1, C_2, M > 0$ and $\beta \geq 0$ are constants. If $\beta = 0$, (\ref{eq:ode-1}) becomes an inequality about a simple linear ODE, and the behavior of $\theta(z)$ follows directly from integrating both sides. If $\beta > 0$, the nonlinear ODE
\begin{equation*}
    \xi'(z) + C_1 \xi(z)^{1+\beta} - C_2 \xi (z) = C, \quad \xi(0)=\xi_0 > 0
\end{equation*}
is solved by 
\begin{equation*}
    \xi(z) = W^{-1} (z)
\end{equation*}
where $W^{-1}$ is the inverse function of 
\begin{equation*}
    W(\xi) = \int_{\xi_0}^\xi \frac{1}{C + C_2 s -C_1 s^{1+\beta}} \mathrm{d}s.
\end{equation*}
Now some fundamental properties of $W(\xi)$ are useful: $W(\xi)$ is monotonically increasing and $W^{-1} (\xi)$ is thus well-defined for $0<\xi<\xi^\ast$ or $\xi>\xi^\ast$, where $\xi^\ast$ is the unique positive solution to $M + C_2 \xi^\ast - C_1 \xi^{\ast(1+\beta)} = 0$; in addition, $W(\xi) \to \infty$ as $\xi \to \xi^\ast+$ or $\xi \to \xi^\ast-$. Thus, it can be seen that we have $\xi(z)$ monotonically increasing if $\xi_0<\xi^\ast$ and $\xi(z)$ monotonically decreasing if $\xi_0>\xi^\ast$, with $\lim_{z\to\infty} \xi(z) = \xi^\ast$ in both cases. Hence, by the comparison theorem for solutions to ODEs \cite[\S 9]{walter2013ordinary}, for the solution to (\ref{eq:ode-1}), we have $\theta(z) \leq \xi(z)$ setting $\theta(0) = \xi_0$, and $\lim_{z\to\infty} \theta(z) \leq \xi^\ast$.

% &C \mu^{2} \frac{\mathrm{d}}{\mathrm{d}z}\|\delta q(z)\|^2 + C_{q_1,\Omega,\text{pc}}h_q^{2} \|\delta q(z)\|^{2(1+\beta)} - \left( \frac{C_{q_1,\Omega,\text{pc}}}{2}h_q^{2}+ \frac{C_{q,b,c,u}}{\mu} + \frac{C_{q,b,c,u}}{\mu^{1+\alpha}} \right)  \|\delta q(z)\|^2\\

%         \leq \,& C_{q,b,c,u}\left(C_{q_1,\Omega,\text{pc}} + \frac{1}{\mu} + \frac{1}{\mu^{1+\alpha}} + \mu^{1+\alpha}  \right) h_q^{2(m-1)}|q|_m^2 + C_{q,b,c,u} \left( C + \frac{1}{\mu} + \frac{1}{\mu^{1+\alpha}} + \mu^{1+\alpha} \right) \left(|\delta b u|_1^2 + \|\delta c u\|^2 + \|\delta f\|^2\right),

The preceding investigation gives us the estimate for $\lim_{z\to\infty} \|\delta
 q(z)\|$ when $\tilde{q}_{h_q} \in V_{h_q}$. The error $\delta I_h u$ in the given data, considered to be included in $\|\delta f\|$ before, is explicitly expressed here. 

\begin{thm}\label{thm-q-err}
    Assume (\ref{eq:positivity-condition}) for some $\beta \geq 0$ and the conditions of Lemma \ref{lemma2} and Lemma \ref{lem6}. Suppose $\tilde{q}_{h_q} \in V_{h_q}$, $q \in H^m(\Omega)$ ($m\geq 1$), $C>0$ to be a constant and $\alpha \in (0, 1)$.\\
    1) If $\beta = 0$ and $C_0 := C_{q_1,\Omega,\text{pc}}\left( 1/2 - 3/\sqrt{\mu} - h\sqrt{\mu} \right) > 0$, then
    \begin{equation*}
        \lim_{z\to\infty} \|\delta q(z)\|^2 \leq C_0^{-1} M h_q^{-2} \left(1 + h \sqrt{\mu} + \frac{1}{\sqrt{\mu}}\right) \left(|\delta b u|_1^2 + \|\delta c u\|^2 + \|\delta f\|^2 +\mu^2 \|\delta I_h u\|^2 \right) + C_0^{-1} M h_q^{2(m-2)} |q|_m^2.
    \end{equation*}
    2) If $\beta > 0$, then we have a monotonically decreasing upper bound for $\|\delta q(z)\|^2$, and
    \begin{equation}\label{eq:est-thm-q-err}
        \lim_{z\to\infty} \|\delta q(z)\|^2 \leq \xi^\ast,
        % C\left( \frac{1}{\mu} + C_A \|u\|_2^2 \mu^{1-\alpha}  + \frac{\|u\|^2}{\mu^{1-\alpha}} \right)C_A^2 h_q^{2(m-1)}|q|_m^2 + \xi^\ast C \left(\frac{1}{\mu} + \mu^{1-\alpha} + \frac{\|u\|^2}{C_{q_1,\Omega,\text{pc}}} \right) \left( |\delta b u|_1^2 + \|\delta c u\|^2 + \|\delta f\|^2 \right),
    \end{equation}
    where $\xi^\ast$ is the unique positive solution to 
    \begin{align*}
        &C_{q_1,\Omega,\text{pc}} \xi^{\ast(1+\beta)} - C_{q_1,\Omega,\text{pc}}\left( \frac{1}{2} +  \frac{3}{\sqrt{\mu}} + h\sqrt{\mu} \right) \xi^\ast\\
        = \,& M h_q^{-2} \left(1 + h \sqrt{\mu} + \frac{1}{\sqrt{\mu}}\right) \left(|\delta b u|_1^2 + \|\delta c u\|^2 + \|\delta f\|^2 + \mu^2 \|\delta I_h u\|^2 \right) + M \left( 1 + h_q^2\left(1+h\sqrt{\mu} + \frac{1}{\sqrt{\mu}}\right) \right) h_q^{2(m-2)} |q|_m^2.
    \end{align*}
    In particular, if $\beta = 1$, writing $ C_{\text{err},1} = h_q^{-2}\left(|\delta b u|_1^2 + \|\delta c u\|^2 + \|\delta f\|^2 + \mu^2 \|\delta I_h u\|^2 \right)$ and $C_{\text{err},2} = h_q^{2(m-2)} |q|_m^2$, we explicitly have that
    \begin{align}\label{eq:est-thm-q-err-1}
    \begin{split}
         \lim_{z\to\infty} \|\delta q(z)\|^2 \leq C &\left(1 + h\sqrt{\mu} + \frac{1}{\sqrt{\mu}}  \right) \\
        + C &C_{q_1,\Omega,\text{pc}}^{-1/2}\left[ C_{\text{err},1} \left(1 + h\sqrt{\mu} + \frac{1}{\sqrt{\mu}}\right) +  C_{\text{err},2} \left(1 + h_q^{2} \left(1 + h\sqrt{\mu} + \frac{1}{\sqrt{\mu}}\right) \right)  \right]^{1/2},
    \end{split}
    \end{align}
    for some $C>0$.
\end{thm}

\subsection{Error estimates for reconstruction of $f$}

We now derive the error estimates for the reconstruction of $f$, following a similar procedure to that of $q$. As in (\ref{eq:idetification-q-discrete}), the spatially discrete scheme for reconstructing $f$ is written as
\begin{equation}\label{eq:idetification-f-discrete}
    \begin{cases}
        - \nabla \cdot (\tilde{q} \nabla v(z)) + \nabla \cdot (\tilde{b} v(z)) + \tilde{c} v(z) = \tilde{f}_{h_f} + \mu \delta I_h u + \mu I_h (u-v(z)),\\
        \int_\Omega \Phi_j \frac{\mathrm{d}}{\mathrm{d} z} \tilde{f}_{h_f} (z) = \frac{1}{\mu} \int_\Omega \Phi_j I_h (u-v(z)), \quad \tilde{f}_{h_f} (0) \in V_{h_f} \text{ the given initial guess},
    \end{cases}
\end{equation}
where $\delta I_h u$ is the error in the data $I_h u$. Here, since we treat the reconstruction error of $f$, the error $\delta I_h u$ cannot be included in $\delta f$ as in Remark \ref{rmk2}, and it must be explicitly expressed. By the Lipschitz continuity (\ref{eq:f-elliptic-lipbound}), (\ref{eq:idetification-f-discrete}) admits a unique solution in $L^2(\Omega)$ for $z\in [0,\infty)$. In parallel to Lemma \ref{lem6}, we have the following lemma.

\begin{lem}\label{lem-deltaf-ineq}
    Suppose $f_{h_f}(z) \in V_{h_f} (z\geq 0)$ satisfies (\ref{eq:idetification-f-discrete}), and $I_{h_f}$ is the $L^2$ projection of $f$ in $V_{h_f}$. Then, with $w(z) = u-v(z)$, $\delta f = \tilde{f}_{h_f} - f$, we have
    \begin{multline*}
        \mu \frac{\mathrm{d}}{\mathrm{d}z} \|\delta f\|^2 + q_0 |w|_1^2 + \mu \|I_h w\|^2 \leq  C \|w\| \left(\|f - I_h f\| +\| \nabla \cdot (\delta q \nabla u) \| + |\delta b u|_1 + \|\delta c u\| + \mu \|\delta I_h u\|\right)\\
        + C h |w|_1 ( \| \delta f \| +\| \nabla \cdot (\delta q \nabla u) \| + |\delta b u|_1 + \|\delta c u\| + \mu \|\delta I_h u\|) + C\|w\| |w|_1,
    \end{multline*}
    where $C>0$ depends on $q_1, b_1, C_P$ and $C_A$.
\end{lem}

\begin{proof}
    From $-\nabla \cdot (\tilde{q} \cdot \nabla w) + \nabla \cdot \tilde{b} w + \tilde{c} w + \mu I_h w + (\tilde{f}_{h_f} - I_{h_f}f) = -\nabla\cdot (\delta q \nabla u) + \nabla \cdot \delta b u + \delta c u + (f - I_{h_f}f) - \mu \delta I_h u$ and (\ref{eq:idetification-f-discrete}), we have
    \begin{align*}
        \int_\Omega \tilde{q} \cdot \nabla w\cdot \nabla I_h w + \nabla \cdot \tilde{b} w I_h w + (\tilde{c}+\mu) I_h w^2 + \delta f_{h_f} I_h w &= \int_\Omega \left(-\nabla\cdot (\delta q \nabla u) + \nabla \cdot \delta b u + \delta c u + f - I_{h_f}f - \mu \delta I_h u\right)I_h w,\\
        \frac{\mu}{2} \frac{\mathrm{d}}{\mathrm{d}z} \|\delta f\|^2 &= \frac{\mu}{2} \frac{\mathrm{d}}{\mathrm{d}z} \|\delta f_{h_f}\|^2 = \int_\Omega \delta f_{h_f} I_h w ,
    \end{align*}
    where $\delta f_{h_f} = \tilde{f}_{h_f} - I_{h_f}f$. Then, proceeding as in Lemma \ref{lem6}, we obtain the conclusion.  
\end{proof}

% We shall need a positivity condition for $u$. As $f$ is assumed to be positive a.e. and $u=0$ on $\partial \Omega$, by weak maximum principle \cite[Section 6.4]{evans2010partial} we have $u > 0$ in $\Omega$, and we further assume
% \begin{equation}\label{eq:positivity-condition-2}\tag{$\text{PC}_2$}
%     u(x) \geq c_{\text{pc}_2} \text{dist}(x, \partial \Omega)^\beta \qquad \forall x\in \Omega.
% \end{equation}
% In the next lemma, we assume further smoothness of $f$. 

For the case of reconstruction of $f$, we do not need to use a positivity condition like (\ref{eq:positivity-condition}) to obtain an (nonlinear) estimate of $\|\delta f\|$ from below, and a linear error estimate is available as the following.

\begin{thm}\label{thm-err-f}
    Assume the conditions of Corollary \ref{coro1}, Corollary \ref{coro2} and Lemma \ref{lem-deltaf-ineq}, and that $f \in H^m(\Omega)$ ($m\geq 0$). Assume also that $1/\sqrt{\mu} \leq q_0$. Then, we have
    \begin{equation}
    \begin{aligned}\label{eq:thm-err-f-ineq-est}
        &C \mu^2 \frac{\mathrm{d}}{\mathrm{d}z} \|\delta f (z)\|^2 + \left( \frac{1}{2} - \frac{1}{4} h \sqrt{\mu} - \frac{C}{\sqrt{\mu}} \right) \|\delta f(z)\|^2 \\
        \leq & \, M_1 h_f^{2 m} |f|_m^2 + M_2 \left( 1 + h \sqrt{\mu} + \frac{1}{\sqrt{\mu}} \right) \left( \|\nabla \cdot (\delta q \nabla u)\|^2 + |\delta b u|_1^2 + \|\delta c u\|^2 + \mu^2 \|\delta I_h u\|^2 \right),
    \end{aligned}
    \end{equation}
    where $M_1$ and $M_2$ depend on $\tilde{q}, q_1, b_1, C_P$ and $C_A$. Thus, if $C_0 := 1/2 - h \sqrt{\mu}/4 - C/\sqrt{\mu} > 0$, the error $\|\delta f (z)\|^2$ decays exponentially, and 
    \begin{equation}
    \begin{aligned}\label{eq:thm-err-f}
        \lim_{z\to \infty} \|\delta f (z)\|^2
        \leq &\,C_0^{-1} \left[M_1 h_f^{2 m} |f|_m^2 + M_2 \left( 1 + h \sqrt{\mu} + \frac{1}{\sqrt{\mu}} \right) \left( \|\nabla \cdot (\delta q \nabla u)\|^2 + |\delta b u|_1^2 + \|\delta c u\|^2 + \mu^2 \|\delta I_h u\|^2 \right) \right].
    \end{aligned}
    \end{equation}
\end{thm}

\begin{proof}
    For the error $\delta f$, we have the relation
    \begin{equation*}
        \delta f = \nabla\cdot(\tilde{q} \nabla w) - \nabla \cdot\tilde{b}w + \tilde{c}w + \mu I_h w -\nabla\cdot (\delta q \nabla u) + \nabla\cdot \delta b u + \delta c u + \mu \delta I_h u.
    \end{equation*}
    Taking $L^2$ inner-product with $\delta f$ for both sides and using (\ref{eq:h2-est}) for the first term on the right side, we have
    \begin{align*}
        \|\delta f\|^2 \leq C \mu \left( \frac{1}{\sqrt{\mu}} |w|_1^2 + \mu \|I_h w\|^2  \right) + \frac{C}{\sqrt{\mu}} \|\delta f\|^2 + C \left( \|\nabla \cdot (\delta q \nabla u)\|^2 + |\delta b u|_1^2 + \|\delta c u\|^2 + \mu^2 \|\delta I_h u\|^2 \right) + \frac{1}{4}\|\delta f\|^2.
    \end{align*}
    Then, using $1/\sqrt{\mu} \leq q_0$ to invoke Lemma \ref{lem-deltaf-ineq} and proceeding as in Lemma \ref{lem-deltaq-ineq}, we obtain (\ref{eq:thm-err-f-ineq-est}). The error estimate (\ref{eq:thm-err-f}) directly follows from (\ref{eq:thm-err-f-ineq-est}).
\end{proof}

\section{Numerical results}\label{sect5}

In this section, we present two numerical examples. In all examples, we take $\Omega = (0,1) \times (0,1)$, and uniform triangle meshes in $\Omega$ are used to build reconstructions of $q$ and $f$ using continuous and piecewise-linear functions. The conjugate gradient method described in \cite[Section 3.4]{hasanouglu2021introduction} is used with suitable fixed step sizes (usually taken to be $\mu$) and the approximated gradients in Section \ref{sec:ellip-aprox-grad} to update the reconstructions $\tilde{q}_{h_q}$ and $\tilde{f}_{h_f}$, and iteration is stopped once the error $J(\tilde{q}, \tilde{f})$ starts to increase. The initial guesses of the coefficients are always taken to be constant functions. The computations are conducted in Wolfram Mathematica 13.2, and, in particular, the numerical solutions of our elliptic equations are obtained using the software's built-in functionalities.

For convenience, in the simulations $I_h u$ is replaced by the quadratic Lagrange interpolation $\mathcal{I}_h u$ that is built from the pointwise data of $u$, which is also a more realistic setting since in practice the data of $u$ are measured at discrete points. And instead of solving (\ref{eq:elliptic-2}) we numerically solve
\begin{equation*}
\begin{aligned}
    -\nabla \cdot ( \tilde{q} \nabla v ) + \nabla \cdot (b v) + c v &= f + \mu (\mathcal{I}_h u - v), &&\text{in } \Omega\\
    v &= 0. &&\text{on } \partial \Omega
\end{aligned}
\end{equation*}
The noisy data, denoted $\mathcal{I}_h u^\delta$, will be produced in two ways. The first way is to demonstrate the stability of the proposed method, and the noisy data of $u$ of a noise level $\delta$ is produced by randomly picking
\begin{equation}\label{eq:obs-error}
    u^\delta (x_j) \in [(1-\delta) u(x_j), (1+\delta) u(x_j)],
\end{equation}
with a uniform probability distribution. The noisy pointwise data are then used to construct a noisy $\mathcal{I}_h u^\delta$. The second way is to demonstrate the error estimates in Theorem \ref{thm-q-err} and Theorem \ref{thm-err-f}, and the noisy data is produced by
\begin{equation}\label{eq:obs-error-2}
    \mathcal{I}_h u^\delta = \mathcal{I}_h \left(1 + \delta \sin (\pi x) \sin (\pi y) \right) u(x, y).
\end{equation}
Here, the "components" of the noise is fixed to be $\sin (\pi x) \sin (\pi y)$ without randomness and we only control the amplitude $\delta$, so the trend of the errors in this case better reveal (\ref{eq:thm-err-f}) and (\ref{eq:est-thm-q-err-1}). In the following examples, we shall denote the relative $L^2$ errors by 
\begin{equation*}
        E_q = \|\tilde{q}_{h_q}^\ast - q\|/\|q\|, \qquad E_f = \|\tilde{f}_{h_f}^\ast - f\|/\|f\|,
\end{equation*}
where $\tilde{q}_{h_q}^\ast$ or $\tilde{f}_{h_f}^\ast$ is a corresponding reconstruction.

\subsection{Example 1}\label{ex1}

In the first example, the coefficients are set to be
\begin{equation*}
    \begin{aligned}
        q &= 2 \sin (3 \pi  x) \sin (2 \pi  y)-0.7 x+5,\\
        b &= \left(2 x^2-y^2+3,\exp \left(x^2\right)-2 y \sin (3 x)+3\right)^\intercal,\\
        c &= \exp (0.2 x+0.3 y)+\cos (2 x) \cos (4 y)+2,\\
        f &= 1.5 \sin (6 x-0.4) \cos (3 y+0.6)+6 (y-1) y+6.\\
    \end{aligned}
\end{equation*}
The numerical solution $u$ of (\ref{eq:elliptic-1}) is solved using the finite element method on a mesh of mesh size $1/1{,}024$ with second-order Lagrange elements. See Figure \ref{fig:ex1-true} for the graph of each coefficient and the solution $u$. A quadratic interpolation of $u$ on a uniform triangle mesh with mesh size $1/16$ is used as $I_h u$. Moreover, for this problem, we verify using the numerical solution of $u$ that the positivity condition (\ref{eq:positivity-condition}) holds with $c_{\text{pc}}=0.05$ and $\beta = 1$.

We first test the reconstruction error in different values of $\mu$ in the two situations that (i) $\tilde{b}$ and $\tilde{c}$ taken the true values, and (ii) very rough estimations of $b$ and $c$ are used as $\tilde{b}$ and $\tilde{c}$, namely $\tilde{b}= (3.5,3.5)^\intercal$ and $\tilde{c} = 3$ (see also Figure \ref{fig:ex1-true}). In both situations, $f$ or $q$ is exactly given, and there is no error added to $I_h u$. The results with mesh sizes $h_q = h_f = 1/8$ are given in Table \ref{tab:ex1-mu-err}. The errors for $q$ and $f$ are both relatively low for a wide range of $\mu$, which can be taken roughly from $100$ to $5{,}000$ for $q$ and from $500$ to $5{,}000$ for $f$. One can explain this phenomenon by Theorem \ref{thm-q-err} and Theorem \ref{thm-err-f}, where the constants in the bound of $L^2$ error contain $h \sqrt{\mu} + 1/\sqrt{\mu}$, which reminds one of the classical Tikhonov regularization (see, for example, \cite{kirsch2011introduction}). As a reference, the relative errors of the direct piecewise-linear Lagrange interpolations are $3.46\times10^{-2}$ for $q$ and $0.97\times10^{-2}$ for $f$. Comparing with them, we see our reconstructions are not even much worse than the direct interpolation of true coefficients when $\tilde{b}$ and $\tilde{c}$ are far from the true values.
\begin{figure}
    \centering
    \includegraphics[width=1\linewidth]{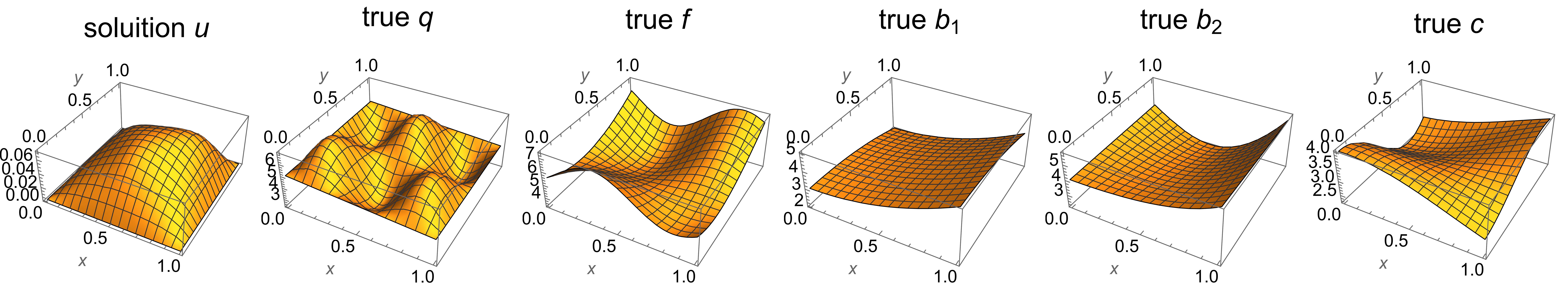}
    \caption{True data for Example 1.}
    \label{fig:ex1-true}
\end{figure}

\begin{table}
    \caption{Error in the reconstruction with $\mu$ varies for Example 1, $h_q = h_f = 1/8$.}
    \centering
    \begin{tabular}{c|ccccccccc}
         \toprule
         \hfill$\mu$&  1.00e1 & 1.00e2 & 2.00e2 & 5.00e2 & 1.00e3 & 5.00e3 & 1.00e4 & 5.00e4\\
         \midrule
         \hfill$E_q (\tilde{b}=b, \tilde{c}=c)$& 9.60e-2 & 3.43e-2  & 2.94e-2 & 2.18e-2 & 2.40e-2 & 3.45e-2 & 5.10e-2 & 1.56e-1\\
         \hfill$E_q (\tilde{b}=(3.5,3.5)^\intercal, \tilde{c}=3)$&  9.53e-2 & 3.64e-2 & 2.81e-2 & 2.70e-2 & 2.95e-2 & 3.50e-2 & 5.23e-2&1.39e-1\\
         \hfill$E_f (\tilde{b}=b, \tilde{c}=c)$ & 6.56e-2& 5.07e-2 & 4.74e-2 & 1.19e-2 & 1.16e-2 & 2.68e-2 & 4.65e-2& 1.44e-1\\
         \hfill$E_f(\tilde{b}=(3.5,3.5)^\intercal, \tilde{c}=3)$ & 6.77e-2& 5.59e-2 & 5.26e-2 & 2.84e-2 & 2.80e-2 & 3.54e-2 &5.12e-2& 1.43e-1\\
         \bottomrule
    \end{tabular}
    \label{tab:ex1-mu-err}
\end{table}

\begin{figure}
    \centering
    \includegraphics[width=.5\linewidth]{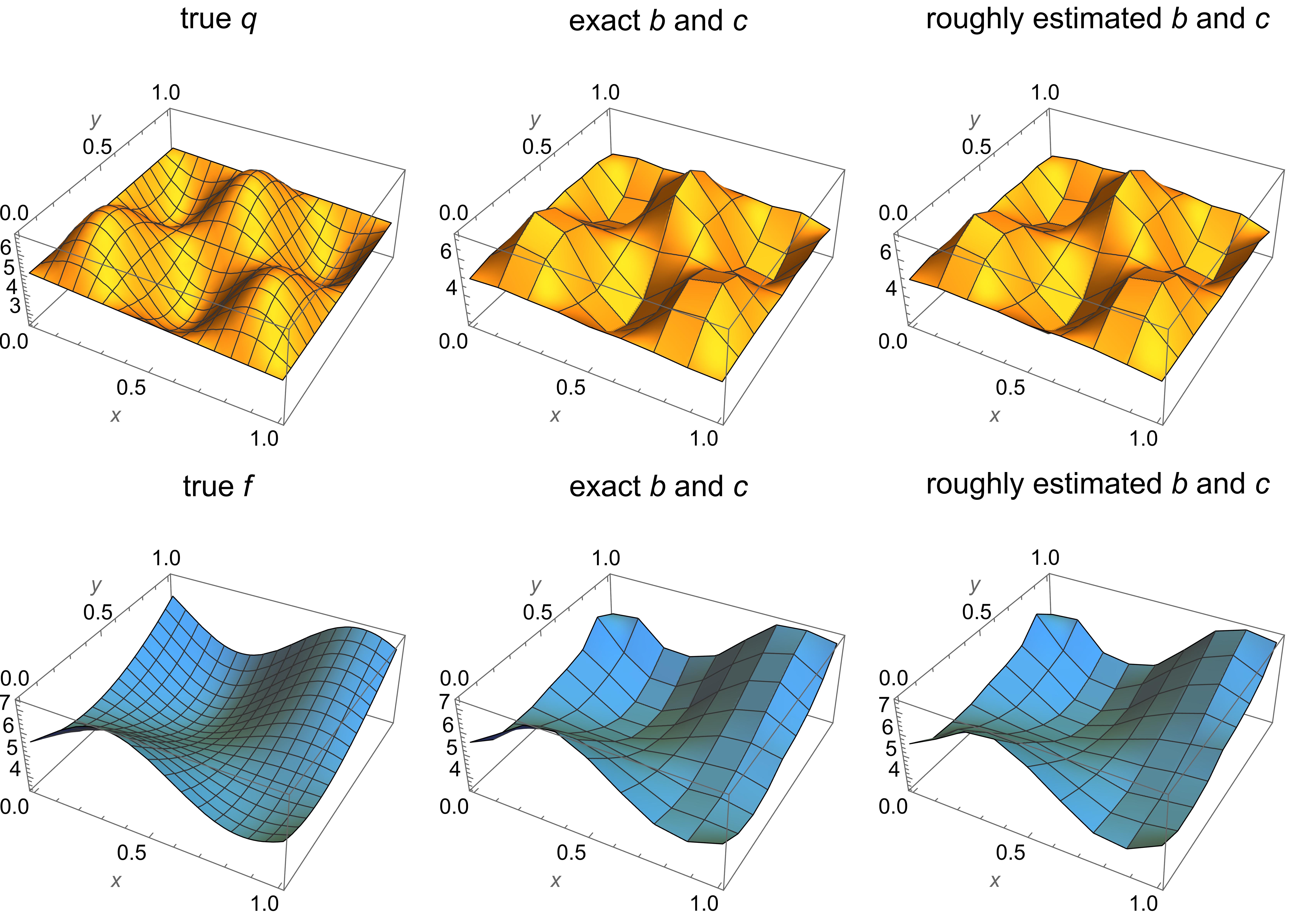}
    \caption{Firs row: the true $q$ and reconstructed $q$ with exact $b$ and $c$ or with $\tilde{b}= (3.5,3.5)^\intercal$ and $\tilde{c} = 3$; second row: the true $f$ and reconstructed $f$ with exact $b$ and $c$ or with $\tilde{b}= (3.5,3.5)^\intercal$ and $\tilde{c} = 3$. $\mu=1{,}000$ in all reconstructions here. The wireframes in the plots of reconstructions reflect the mesh sizes $h_q = h_f = 1/8$, while those in the plots of true coefficients are for the clarity of their profiles.}
    \label{fig:ex1-rec-q-f-mu}
\end{figure}
Next, we test the reconstruction error for different levels $\delta$ of random noise in $\mathcal{I}_h u^\delta$ produced by (\ref{eq:obs-error}). We use $h_q = h_f = 1/4$ and $\mu=1{,}000$ for the reconstruction of $q$ and $\mu = 1{,}200$ for the reconstruction of $f$ in computations with $\delta$ increasing from $0$ to $0.1$, or, $10\%$. The results are shown in Figure \ref{fig:ex1-err-delta}. We see that the decay of the reconstruction errors of $q$ and $f$ is basically linear as $\delta\to0$. This shows the stability of the reconstruction method, and the almost linear decay of the reconstruction errors for $q$ is more optimal than the sub-linear decay predicted by Theorem \ref{thm-q-err}.

% \begin{table}
%     \caption{Error in the identification under different noise level of $I_h u$ for Example 1, $h_q = h_f = 1/4$.}
%     \centering
%     \begin{tabular}{c|ccccccccccc}
%          \toprule
%          \hfill$\delta$&  0.00e0 & 1.00e-2 & 2.00e-2 & 3.00e-2 & 4.00e-2 & 5.00e-2 & 6.00e-2 & 7.00e-2 & 8.00e-2 & 9.00e-2 & 1.00e-1\\
%          \midrule
%          \hfill$E_q$ ($\mu=1000$)&  7.04e-2 & 7.51e-2 & 8.48e-2 & 9.89e-2 & 1.21e-1 & 1.41e-1& 1.97e-1 & 1.77e-1 & 2.27e-1 & 2.53e-1 & 2.75e-1 \\
%          \hfill$E_f$($\mu=1200$)&  2.06e-2 & 2.87e-2 & 3.15e-2 & 4.90e-2 & 5.16e-2 & 6.73e-2 & 6.87e-2 & 8.81e-2 & 7.98e-2 & 9.19e-1 & 9.00e-2\\
%          \bottomrule
%     \end{tabular}
%     \label{tab:ex1-delta-err}
% \end{table}

\begin{figure}
    \centering
    \includegraphics[width=.75\linewidth]{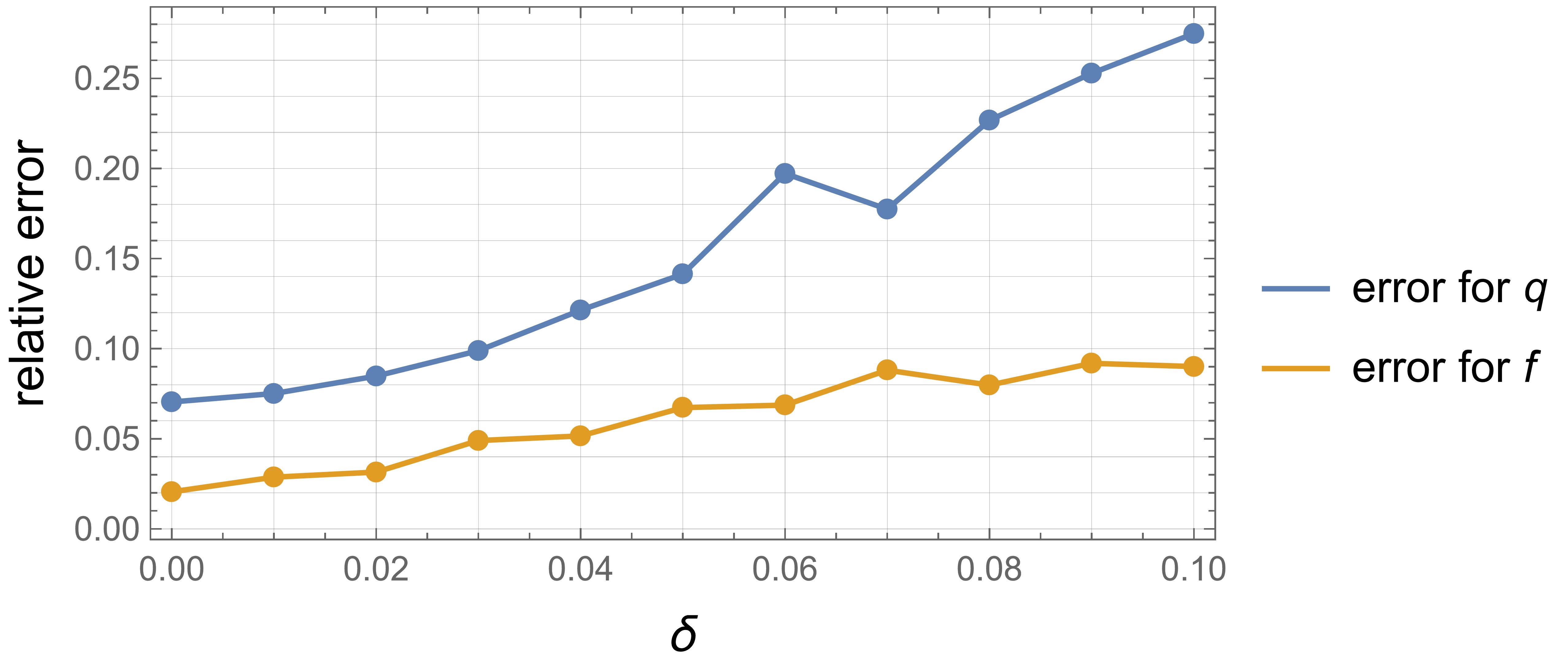}
    \caption{Relative $L^2$ errors for $q$ and $f$ with noisy data produced by (\ref{eq:obs-error}), Example 1.}
    \label{fig:ex1-err-delta}
\end{figure}

Then, we consider the noisy data $\mathcal{I}_h u^\delta$ produced by (\ref{eq:obs-error-2}). The mesh size $h$ of the observation is taken to be $1/128$ in this case to suppress the discretization error in $\mathcal{I}_h u$, and the mesh sizes $h_q$ and $h_f$ for reconstruction are both set to be $1/4$. The amplitude of the noise $\delta$ is taken from $0$ to $0.65$ with a spacing of $0.05$, and the parameter $\mu$ in computations ranges from $1$ to $5{,}000$. The results are shown in Figure \ref{fig:ex1-sin-err}, where on the horizontal axis we denote $\delta_{L^2} = \|\mathcal{I}_h u^\delta - \mathcal{I}_h u\|/\|\mathcal{I}_h u\|$ as the relative $L^2$ error in the given data. Sub-linear decay rates are seen in the results for $q$, while the decay rates in the results for $f$ are basically linear, especially for large values of $\mu$. This demonstrates the sub-linear error estimate in Theorem \ref{thm-q-err} and the linear error estimate in Theorem \ref{thm-err-f}. In addition, the behaviors of the method with different values of $\mu$ are clearly exhibited in Figure \ref{fig:ex1-sin-err}, which show that relatively large values of $\mu$ produce much better reconstructions under a reasonable noise level. However, Theorem \ref{thm-q-err} with $\beta = 1$ only gives a sub-optimal error estimate in this example. The decay rates of errors for $q$ in Figure \ref{fig:ex1-sin-err} are approximately $0.6$ for large $\mu$, while Theorem \ref{thm-q-err} with $\beta = 1$ predicts $1/(1+\beta) = 0.5$ (see (\ref{eq:est-thm-q-err-1})). Comparing Figure \ref{fig:ex1-sin-err} with Figure \ref{fig:ex1-err-delta}, we also see that the decay of the error with respect to the data error behaves almost linearly except for considerably large data errors.

\begin{figure}
    \centering
    \includegraphics[width=.85\linewidth]{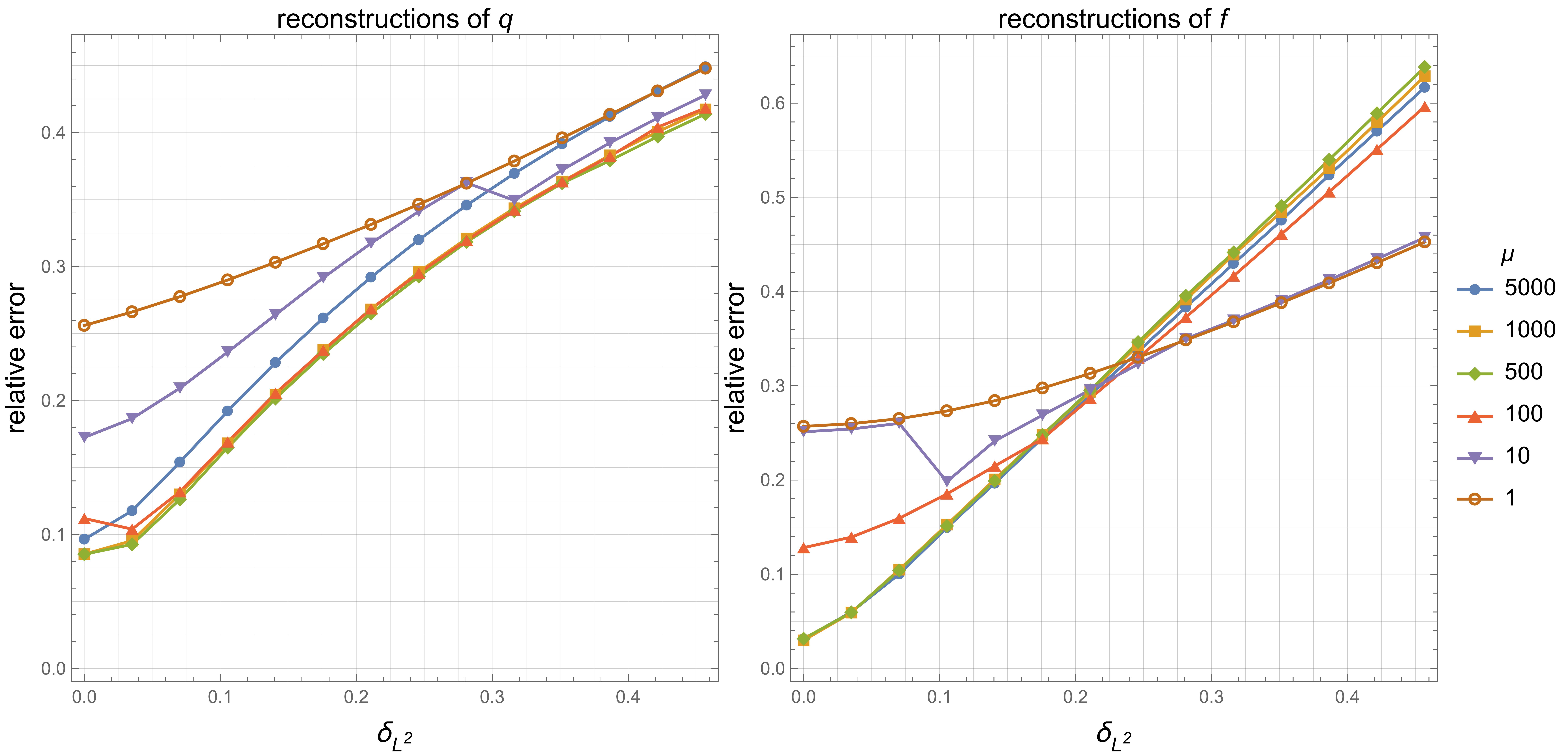}
    \caption{Relative reconstruction errors for $q$ and $f$ with noisy data produced by (\ref{eq:obs-error-2}), Example 1.}
    \label{fig:ex1-sin-err}
\end{figure}

However, in general, the finer meshes used for the reconstruction do not work very well, which reveals the ill-posedness of the problem. For example, with a mesh size of $h_q = h_f = 1/16$ and $\mu = 500$ with exact $b$, $c$, $f$ and $\mathcal{I}_h u$ ($h=1/16$), the reconstruction of $q$ is oscillatory, and the the reconstruction of $f$ inexact near the boundary of the domain; see Figure \ref{fig:ex1-q-16}. Here, referring to Theorem \ref{thm-q-err}, we see that the increase of $h_q^{-1}$ results in the growth of $\xi^\ast$ in (\ref{eq:est-thm-q-err}). For the reconstruction of $f$, in Theorem \ref{thm-err-f} the estimated reconstruction error bound for $f$ is not affected by $h_f^{-1}$ (recall that no inverse estimate is used to prove Theorem \ref{thm-err-f}), and the reconstruction is smooth in the inner part of the domain. However, the accuracy near the boundary is not satisfactory. The corresponding $L^2$ relative errors are $E_q = 4.40\times 10^{-2}$ and $E_f = 6.28 \times 10^{-2}$, which are not better than the result when $h_q = h_f = 1/8$. 

\begin{figure}
    \centering
    \includegraphics[width=.8\linewidth]{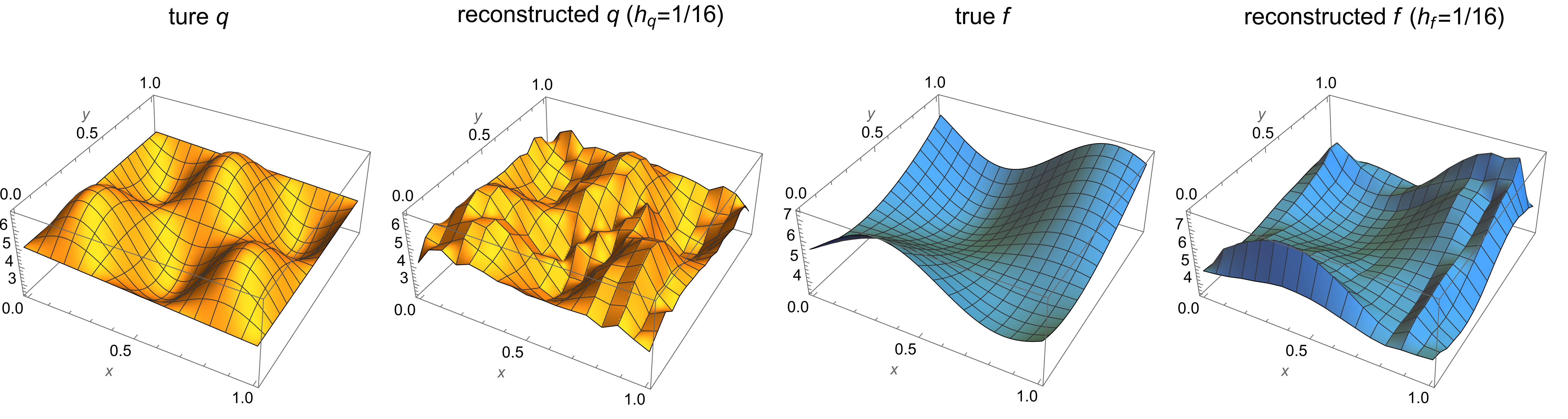}
    \caption{True coefficients and the reconstructed coefficients using $h_q = h_f = 1/16$, Example 1.}
    \label{fig:ex1-q-16}
\end{figure}

\subsection{Example 2}

In this second example, we set the following data:
\begin{equation*}
    \begin{aligned}
        q &= \sin (5 y)+4 + \begin{cases}
            20 x^2+1 & x<0.4 \\
            15 (x-1)^2-1.2 & x\geq 0.4\end{cases},\\
        b &= \left(-2 \sin ^2(2 x)+y^2+4, -x^2-2 y \sin (3 y)+4\right)^\intercal,\\
        c &= (x-1)^2+\sin (2 x)+y^2+1,\\
        f &= 5 \exp \left(-10 (x-0.5)^2-10 (y-0.5)^2\right)+1.8 \exp (-x) \exp (0.5 y)+2.
        \end{aligned}
\end{equation*}
See Figure \ref{fig:ex2-true} for their profiles. The numerical solution of $u$ and its interpolation $\mathcal{I}_h u$ are produced with the same procedure as in the first example. Note that $q$ is not smooth on $x=0.4$ in this example. For this problem, positivity condition (\ref{eq:positivity-condition}) holds with $c_{\text{pc}}=0.05$ and $\beta = 1$.

The reconstruction errors for $q$ and $f$ in different values of $\mu$ are given in Table \ref{tab:ex2-mu-err}. $h_q$ and $h_f$ are set to be $1/8$, with inputted $\tilde{b} = b$, $\tilde{c} = c$ and $\tilde{b} = (2.5, 2.5)^\intercal$, $\tilde{c} = 2.8$, respectively. As a reference, first-order Lagrange interpolation errors for $q$ and $f$ are respectively $2.41 \times 10^{-2}$ and $1.45 \times 10^{-2}$. In this example, we also see that reconstructions using very rough estimates of $b$ and $c$ can still yield satisfying accuracy. Figure \ref{fig:ex2-rec-q-f-mu} shows the graphs of the true coefficients and the reconstructions. As can be seen, there is almost no difference between the reconstructions using the exact $b$ and $c$ and those using only rough estimations of $b$ and $c$.

\begin{figure}
    \centering
    \includegraphics[width=1\linewidth]{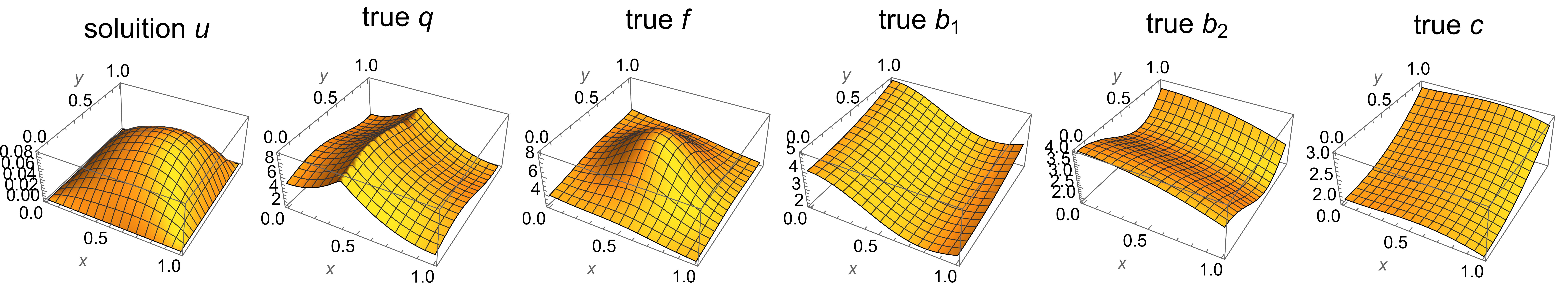}
    \caption{True data for Example 2.}
    \label{fig:ex2-true}
\end{figure}

\begin{table}
    \caption{Error in the reconstruction with $\mu$ varies for Example 2, $h_q = h_f = 1/8$.}
    \centering
    \begin{tabular}{c|ccccccccc}
         \toprule
         \hfill$\mu$& 1.00e1 & 1.00e2 & 2.00e2 & 5.00e2 & 1.00e3 & 5.00e3 & 1.00e4 & 5.00e4\\
         \midrule
         \hfill$E_q (\tilde{b}=b, \tilde{c}=c)$& 6.37e-2 & 4.71e-2  & 1.83e-2 & 1.83e-2 & 1.72e-2 & 2.08e-2 & 2.51e-2 & 5.68e-1\\
         \hfill$E_q (\tilde{b}=(2.5,2.5)^\intercal, \tilde{c}=2.8)$& 7.67e-2 & 6.30e-2 & 4.83e-2 & 4.70e-2 & 4.66e-2 & 4.77e-2 & 5.17e-2&8.94e-1\\
         \hfill$E_f (\tilde{b}=b, \tilde{c}=c)$ & 1.63e-1& 8.70e-2 & 5.73e-2 & 1.64e-2 & 1.71e-2 & 1.79e-2 & 2.13e-2& 1.44e-1\\
         \hfill$E_f(\tilde{b}=(2.5,2.5)^\intercal, \tilde{c}=2.8)$ & 1.63e-1& 9.50e-1 & 7.35e-2 & 5.92e-2 & 5.92e-2 & 5.99e-2 &5.91e-2& 9.62e-1\\
         \bottomrule
    \end{tabular}
    \label{tab:ex2-mu-err}
\end{table}

\begin{figure}
    \centering
    \includegraphics[width=.5\linewidth]{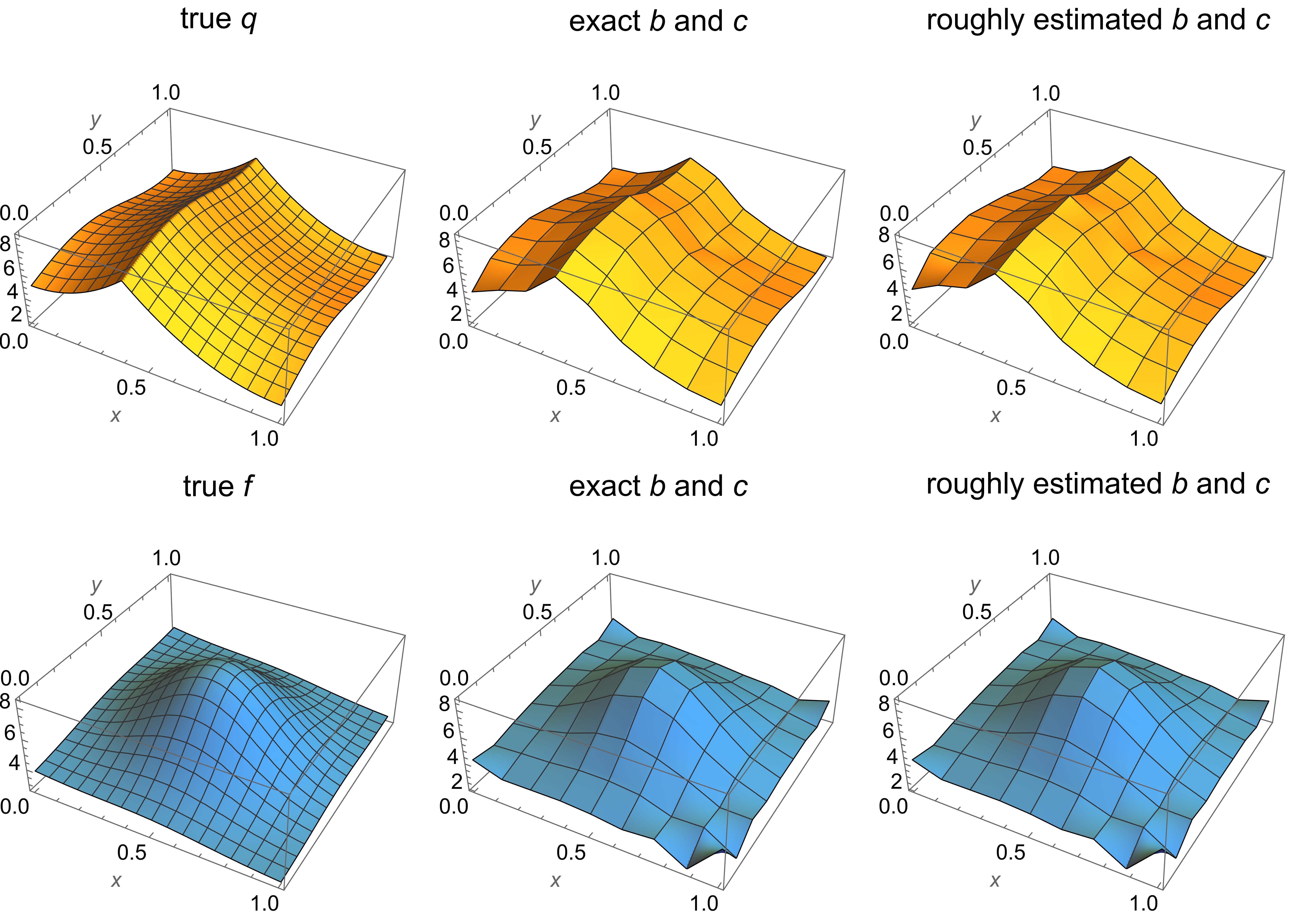}
    \caption{Firs row: the true $q$ and reconstructed $q$ with exact $b$ and $c$ or with $\tilde{b}= (2.5,2.5)^\intercal$ and $\tilde{c} = 2.8$; second row: the true $f$ and reconstructed $f$ with exact $b$ and $c$ or with $\tilde{b}= (2.5,2.5)^\intercal$ and $\tilde{c} = 2.8$. $\mu = 1{,}000$ in all the reconstructions here. The wireframes in the plots of reconstructions reflect the mesh sizes $h_q = h_f = 1/8$.}
    \label{fig:ex2-rec-q-f-mu}
\end{figure}

Reconstruction errors are also investigated at different levels $\delta$ of random noise in $\mathcal{I}_h u$ produced by (\ref{eq:obs-error}), and the results are shown in Figure \ref{fig:ex2-err-delta}. The mesh sizes are $h_q = h_f = 1/4$, and $\mu=500$ for $q$ and $\mu=1{,}000$ for $f$. As shown in Figure \ref{fig:ex2-err-delta}, the reconstruction errors decay basically linearly with the decrease of the noise level $\delta$, indicating the stability of the method.

% \begin{table}
%     \caption{Error in the identification under different noise level of $I_h u$ for Example 1, $h_q = h_f = 1/4$, Example 2.}
%     \centering
%     \begin{tabular}{c|ccccccccc}
%          \toprule
%          \hfill$\delta$&  0.00e0 & 1.00e-2 & 2.00e-2 & 3.00e-2 & 4.00e-2 & 5.00e-2 & 6.00e-2 & 7.00e-2 & 8.00e-2 \\
%          \midrule
%          \hfill$E_q$ ($\mu=1000$)&  7.88e-2 & 7.93e-2 & 8.48e-2 & 9.60e-2 & 1.18e-1& 1.14e-1 & 1.50e-1 & 1.81e-1 & 1.98e-1 \\
%          \hfill$E_f$($\mu=1000$)&  3.11e-2 & 3.75e-2 & 4.74e-2 & 4.92e-2 & 7.58e-2 & 6.55e-2 & 9.89e-2 & 1.04e-1 & 1.11e-1\\
%          \bottomrule
%     \end{tabular}
%     \label{tab:ex2-delta-err}
% \end{table}

\begin{figure}
    \centering
    \includegraphics[width=.75\linewidth]{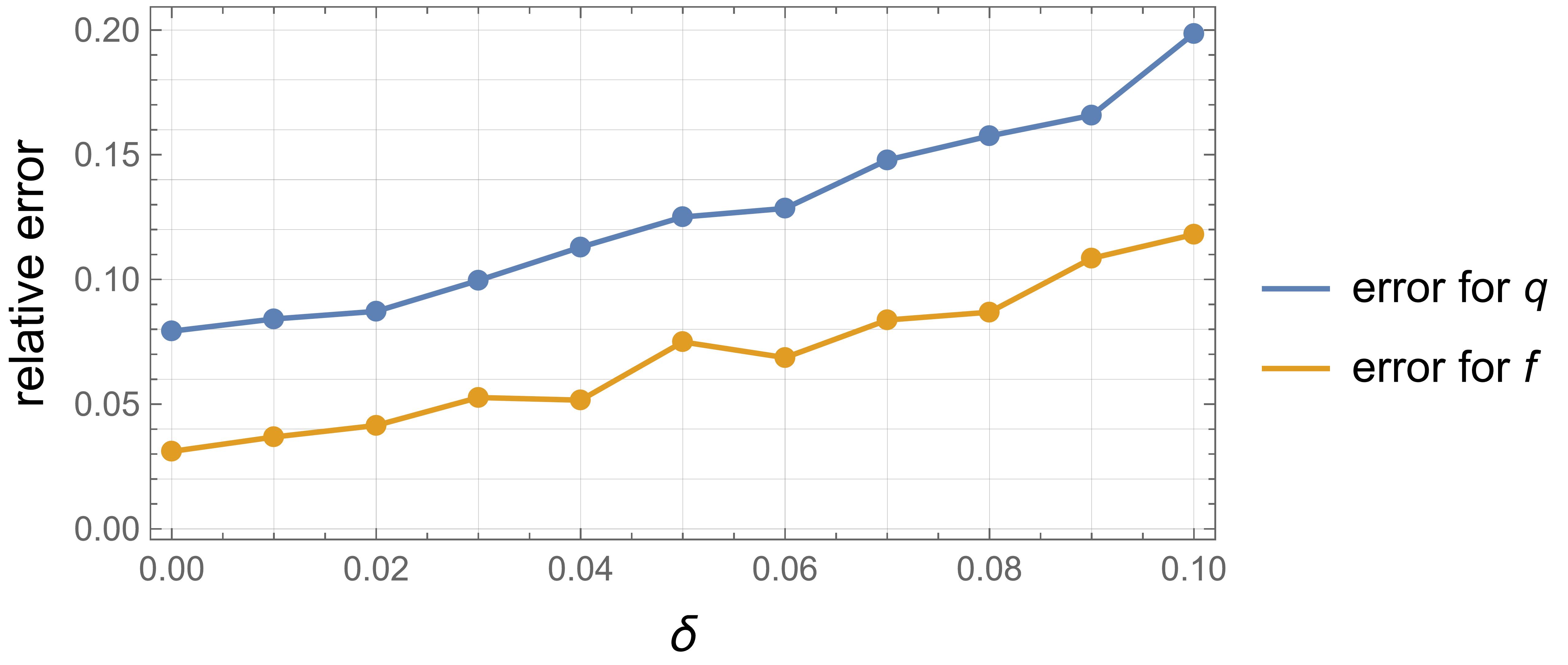}
    \caption{Relative $L^2$ errors for $q$ and $f$ with noisy data produced by (\ref{eq:obs-error}), Example 2.}
    \label{fig:ex2-err-delta}
\end{figure}

The errors with noisy data produced by (\ref{eq:obs-error-2}) are shown in Figure (\ref{fig:ex1-sin-err-2}). The comments for Figure \ref{fig:ex1-sin-err} apply here as well, except that the reconstructions of $f$ using small values of $\mu$ ($1$ and $10$) are relatively accurate, although still with greater errors than large values of $\mu$, in this case. 

\begin{figure}
    \centering
    \includegraphics[width=.85\linewidth]{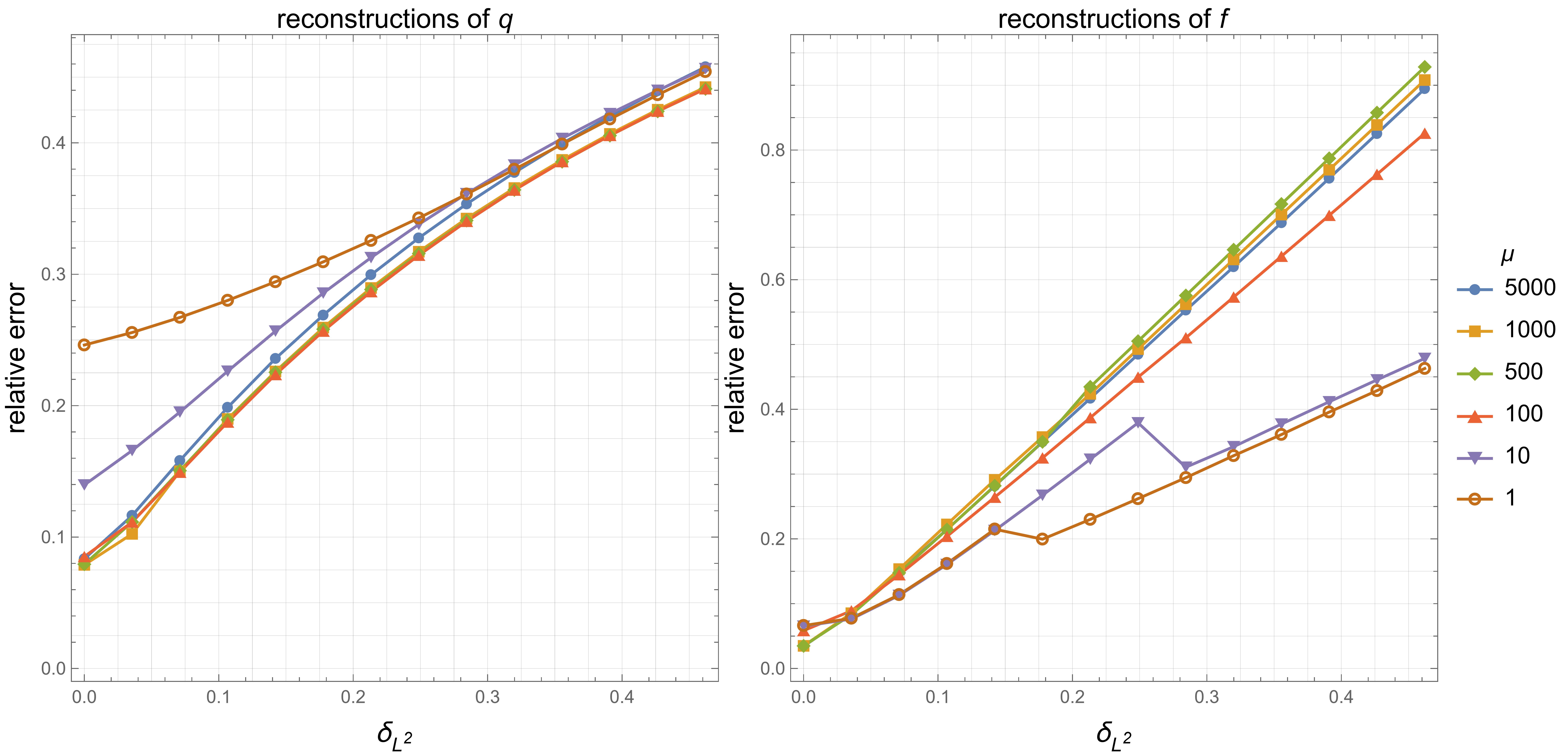}
    \caption{Relative reconstruction errors for $q$ and $f$ with noisy data produced by (\ref{eq:obs-error-2}), Example 2.}
    \label{fig:ex1-sin-err-2}
\end{figure}

\section{Conclusions}

In this work, we present a method for reconstructing coefficients (and the source term) in the elliptic partial differential equation. The error estimates are also derived and verified by numerical experiments. These estimates provide guidelines for selecting parameters in the computations for the reconstruction problem. In addition, the numerical results show that the proposed method performs well in the face of inaccurately inputted coefficient data, which is particularly useful when little knowledge is available about those coefficients.

We are also interested in some further work. For example, how to simultaneously identify $q$ and $f$ or other coefficients; the application of the method to nonlinear elliptic equations; how to refine the technique to make it work better on fine meshes for reconstruction. It is also natural to apply the technique to parabolic equations and investigate the accuracy of reconstruction, as well as the analogues of Theorem \ref{thm-q-err} and Theorem \ref{thm-err-f}.

\bibliographystyle{elsarticle-num}
% \bibliography{ref}

% \printbibliography

\end{document}